\documentclass{amsart}
\usepackage{amsmath}
\usepackage[utf8]{inputenc}
\usepackage{lmodern}
\usepackage[T1]{fontenc}
\usepackage{amssymb}
\usepackage{mathtools}

\usepackage{bm}

\usepackage[shortlabels]{enumitem}

\usepackage[final,hidelinks]{hyperref} 

\usepackage{mathrsfs}
\usepackage[all]{xy}
\usepackage{color}

\usepackage{colonequals}

\usepackage{musicography} 

\usepackage{tikz-cd}

\newtheorem{theorem}{Theorem}[section]
\newtheorem{proposition}[theorem]{Proposition}
\newtheorem{lemma}[theorem]{Lemma}
\newtheorem{corollary}[theorem]{Corollary}

\theoremstyle{definition}
\newtheorem{definition}[theorem]{Definition}
\newtheorem{bigremark}[theorem]{Remark}

\newtheorem{example}[theorem]{Example}

\newtheorem{maintheorem}{Theorem}

\newtheorem{mainthm}{Theorem}

\hypersetup{
    colorlinks=true, 
    linkcolor=blue, 
    urlcolor=red, 
    linktoc=all 
}

\newcommand{\Z}{\ensuremath{\mathbb{Z}}}
\newcommand{\C}{\ensuremath{\mathbb{C}}}

\def\FF{\mathbb{F}}
\def\NN{\mathbb{N}}

\def\Nrv{\mathsf{Nrv}}

\numberwithin{equation}{section} 

\def\kk{\Bbbk}
\def\xx{x}
\def\yy{y}

\def\Sh{\mathrm{Sh}}

\def\simgp{G}
\def\G{\mathcal{G}}
\def\p{\mathsf{p}}
\def\q{\mathsf{q}}

\def\mkd#1{\mathfrak{X}(#1)}
\def\bi{\bm{i}}

\def\C{\mathsf{C}}
\def\N{\mathsf{N}}
\def\A{\mathsf{A}}

\def\sz{\mathrm{Sz}}
\def\hatsz{\hat{\mathrm{S}}\mathrm{z}}
\def\sk{\mathrm{skel}}

\def\sgn{\mathrm{sgn}}

\def\hom{\mathrm{Hom}}

\newcommand*\cocolon{%
        \nobreak
        \mskip6mu plus1mu
        \mathpunct{}%
        \nonscript
        \mkern-\thinmuskip
        {:}%
        \mskip2mu
        \relax
}

\def\id{\mathrm{id}}

\def\PH{\mathrm{PH}}

\title[Szczarba's Twisted Shuffle and Equivariant Path Homology]{Szczarba's Twisted Shuffle and Equivariant Path Homology of Directed Graphs}

\author[X. Fu]{Xin Fu}
\address{Shanghai Institute for Mathematics and Interdisciplinary Sciences (SIMIS), Shanghai 200433, China}
\email{x.fu@simis.cn}
\thanks{X.\,Fu is supported by the National Natural Science Foundation of
China (grant no. 12501083) and the startup fund of SIMIS}

\author[S.-T. Yau]{Shing-Tung Yau}
\address{Yau Mathematical Sciences Center, Tsinghua University, Beijing 100084, China}
\email{styau@tsinghua.edu.cn}



\begin{document}
\begin{abstract}
To a marked simplicial set one can associate its path chain complex, and define its homology to be the homology of this complex, inspired by path homology theories for directed graphs, quivers, and marked categories. Given a marked simplicial set with a simplicial group action preserving the markings and degenerate 1-simplices, together with a twisting function,
we define a marked twisted Cartesian product using the box product.

Classically, Szczarba’s twisted shuffle provides a quasi-isomorphism between the chain complex of a twisted Cartesian product and the corresponding twisted tensor product. 
In this paper, we prove that in the marked setting, this map restricts to a chain isomorphism on path chain complexes. As an application,
for directed graphs with group actions, we obtain a natural Borel construction as a special case of marked twisted Cartesian products. Equivariant path homology is defined as the homology of this construction and is computed by an explicit twisted tensor product.

\end{abstract}
 \maketitle

\section{Introduction}

In recent years several homology theories for directed graphs have appeared. 
The~GLMY path homology, introduced by Grigor’yan–Lin–Muranov–Yau (with cohomological precursors due to Dimakis–Müller-Hoissen~\cite{DimakisHoissen2, DimakisHoissen1}), is constructed from allowable directed paths and satisfies homotopy invariance, the Eilenberg–Steenrod axioms and K\"unneth theorems with respect to the box product (\cite{GJMY, GLMY, GLMY2014}).
Another theory is provided by magnitude homology, defined by Hepworth and Willerton~\cite{HW2017} as a categorification of the graph invariant \emph{magnitude}~\cite{Leinster2019}. 

Although developed independently, Asao~\cite{Asao2023} related the two homologies via the \emph{Magnitude–Path Spectral Sequence} or \emph{MPSS} whose 
$E_1$-page is magnitude homology and whose $E_2$-page contains the path homology.
The target of MPSS, named \emph{reachability homology} by Hepworth and Roff~\cite{HR2025}, is itself a homology theory for directed graphs. In the sequel~\cite{HR2024},
the entire $E_2$-page is termed the \emph{bigraded path homology} and its properties are analyzed.

From a marked simplicial perspective, Ivanov and Pavutnitskiy~\cite{ivanov2024simplicial} introduced a framework for path-based homology via path pairs, situating GLMY path homology in a broader simplicial setting and producing new homology theories, including those for marked categories and marked groups.
Other developments include singular homology theories for quivers~\cite{LMWY} and cellular complex for directed graphs~\cite{TY}.

Many directed graphs admit nontrivial symmetries, and it is therefore natural to ask how to construct equivariant versions of homology that incorporate group actions. 
In classical algebraic topology, equivariant homology theories can be defined by the
Borel construction. In this paper, we propose an analogue of the Borel construction suitable for directed graphs and GLMY path homology.
Our approach builds on the marked simplicial viewpoint.
A marked simplicial set chooses certain $1$-simplices of a simplicial set to be ``marked''.
This perspective leads naturally to the notion of marked twisted Cartesian products, extending both twisted Cartesian products and Borel constructions to the marked setting. 

Concretely, we say a simplicial group $\simgp$ acts on a marked simplicial set $(F,M)$ if~$G$ acts on $F$ and the action preserves $M$. Moreover, if $gs_0(x)\in F_1$ is degenerate for any $g\in G_1$ and $x\in F_0$, a twisting function $\tau\colon X_{>0}\to \simgp$ allows one to define a marked twisted Cartesian product, denoted by
$X^{\musSharp{}}\Box_\tau (F,M).$
Its underlying simplicial set is the twisted Cartesian product $X\times_\tau F$ and the marked edges are given by~$s_0(X_0)\times M\cup X_1\times s_0(F_0).$
See details in Section~\ref{section marked twisted Cartesian products}.

In the unmarked case, Szczarba~\cite{Szczarba1961} gave an explicit formula of the twisting cochain $t_{\sz}\colon \C(X)\to \C(\simgp)$ by the twisting function~$\tau$. This yields a twisted tensor product $C(X)\otimes_{t_{\sz}}\C(F)$. 
He further constructed a chain map 
\[\psi\colon C(X)\otimes_{t_{\sz}}\C(F)\to \C(X\times_\tau F)\]
and proved that $\psi$ is a chain homotopy equivalence. This map $\psi$ is referred to as Szczarba’s twisted shuffle map (see also~Subsection~\ref{subsection Szczarba's twisting cochain}). 

In this paper we establish a marked analogue. After passing $\psi$ to normalised complexes, we obtain the following result.

\begin{maintheorem}[Theorem~\ref{theorem Szczarba's twisted shuffle map}]\label{main theorem}
Suppose $gs_0(x)\in F_1$ is degenerate for any $g\in G_1$ and $x\in F_0$.
The Szczarba's twisted shuffle map on normalised chain complexes restricts to an isomorphism of chain complexes 
\[ \N(X)\otimes_{t_{\sz}}\Omega_\bullet(F,M)\to \Omega_\bullet(X^{\musSharp{}}\Box_\tau (F,M)).
\]
\end{maintheorem}

This result provides a framework for studying equivariant path homology of directed graphs. 
A directed graph $\mathcal{G}=(V,E)$ gives rise to a marked simplicial set 
\[(\Nrv(\mathcal{G}), E),\]
where $\Nrv(\mathcal{G})$ is the nerve of the preorder $(V,\leq)$ defined by $u\leq v$ whenever there is a directed path from $u$ to $v$, and the marked edges are given by the directed edges. 
The associated path chain complex of $(\Nrv(\mathcal{G}), E)$ agree with the path chain complex $\Omega_\bullet(\mathcal{G})$ introduced by Grigor’yan–Lin–Muranov–Yau.

If a group ~$\Gamma$ acts on $\mathcal{G}$, then the associated marked simplicial set carries an induced action preserving the marked edges. In this setting the Borel construction 
\[
E\underline{\Gamma}\times_{\underline{\Gamma}} \Nrv(\mathcal{G}),
\]
where $\underline{\Gamma}$ denotes the constant simplicial group on $\Gamma$, becomes a marked simplicial set whose marked edges are given by the $\Gamma$-orbits of
\[\big((E\underline{\Gamma})_1\times E^D\big)\cup \big(s_0((E\underline{\Gamma})_0)\times E\big),\] where $E^D=\{(v,v)\mid v\in V\}$ and $s_0$ is the degeneracy map of the simplicial set~$E\underline{\Gamma}$.
This marked Borel construction is denoted by $\mathcal{G}_\Gamma$, and its path homology defines the~$\Gamma$-equivariant path homology of $\mathcal{G}$. 

Theorem~\ref{main theorem} provides an explicit model for the equivariant path homology of a
directed $\Gamma$-graph.

\begin{mainthm}[Proposition~\ref{prop model for equivariant homology digraph}]
Let $\mathcal{G}=(V,E)$ be a 
directed $\Gamma$-graph and  $\Omega_\bullet(\mathcal{G})$ be its GLMY path chain complex. Then there is a quasi-isomorphism
\[
\N(B\Gamma)\otimes_{t_{\sz}}\Omega_\bullet(\mathcal{G})\to\Omega_\bullet(\mathcal{G}_\Gamma),
\]
which is natural with respect to any $\Gamma$-equivariant inclusion $\mathcal{G}'\subset \mathcal{G}$ of directed graphs.
\end{mainthm}

\section{Preliminaries}
Throughout this paper, $\kk$ will be a principal ideal domain, and all tensor products and chain complexes are taken over~$\kk$, unless specified otherwise.

\subsection{Simplicial sets}

A \emph{simplicial set} $X$ is a graded set indexed on non-negative integers $n$, together with maps
$\partial_i\colon X_n\to X_{n-1}$ and $s_i\colon X_n\to X_{n+1}$
for $0\leq i\leq n$, satisfying the simplicial identities; see~\cite[Definition 1.1]{May}. 
We write $\mathsf{sSet}$ for the category of simplicial sets. 
If $x=s_iy$ for some $y\in X_{n-1}$ and some map $s_i$, then $x$ is called \emph{degenerate}; otherwise $x$ is called \emph{non-degenerate}.

Given $X\in \mathsf{sSet}$, 
one may consider three chain complexes $\mathsf{C}(X)$, $\mathsf{N}(X)$ and $\mathsf{D}(X)$. The $n$-th component of $\mathsf{C}(X)$ is a free $\kk$-module generated by $X_n$ and the differential
\[
\partial \colon \mathsf{C}_n(X)\xrightarrow{} \mathsf{C}_{n-1}(X)
\]
is defined by $\partial(x) =\sum_{i=0}^n(-1)^i\partial_i(x)$ for $x\in X_n$. 
The complex $\mathsf{D}(X)$ is a subcomplex of~$\mathsf{C}(X)$ generated by degenerate simplices. The normalised complex or Moore complex $\mathsf{N}(X)$ has $n$-th component given by 
\[
\mathsf{N}_n(X)=\bigcap_{i\neq 0}\ker \left(\partial_i\colon \C_n(X)\to \C_{n-1}(X)\right)
\]
with differential induced by $\partial_0.$
There is a decomposition of chain complexes
\[
\mathsf{C}(X)= \mathsf{D}(X)\oplus \mathsf{N}(X).
\]
Hence we often identify $\N(X)$ with $\mathsf{C}(X)/\mathsf{D}(X)$, and write $\rho\colon \mathsf{C}(X)\xrightarrow{} \mathsf{N}(X)$ for the projection map.

\subsection{Twisted tensor products}

Let $C$ be a be a coaugmented differential graded coalgebra (dgc), and let $A$ be an augmented differential graded algebra (dga). The complex $\hom(C,A)$ is a dga via
\begin{align*}
d(f)&=d_A\circ f-(-1)^{\deg f}f\circ d_C,&\text{(differential)}\\
f\cup g&=m_A(f\otimes g)\Delta_C,&\text{(product)} 
\end{align*}
where $m_A$ denotes the multiplication on $A$, and $\Delta_C$ is the comultiplication on $C$.

\smallskip

A \emph{twisting cochain} is an element $t\in \hom(C,A)$ of degree $-1$ such that
\[
d(t)=t\cup t,\quad t \,\iota_C=0,\quad \varepsilon_A\, t=0,
\]
where $\iota_C\colon \kk \to C$ is the coaugmentation and $\varepsilon_A\colon A \to \kk$ is the augmentation.

Let $M$ be an augmented differential graded $A$-module, with action map 
\[\mu_M\colon A \otimes M \to M.\]
The \emph{twisted tensor product} $C \otimes_t M$ is the chain complex whose underlying graded $\kk$-module is $C \otimes M$, and whose differential is given by
\begin{equation}\label{eq twisted tensor product differential}
d_t=d_C\otimes 1+1\otimes d_M-(1\otimes \mu_M)(1\otimes t\otimes 1)(\Delta_C\otimes 1).
\end{equation}

\subsection{Twisted Cartesion products}

Let $X$ be a simplicial set and $G$ a simplicial group.
A \emph{twisting function} $\tau\colon X\to G$ is a sequence of maps $\tau_n\colon X_n\to G_{n-1}$ for~$n\geq 1$ such that
\begin{align}
\partial_0\tau(x)&=\tau(\partial_0x)^{-1}\cdot \tau(\partial_1 x); \\
\partial_i\tau(x)&=\tau(\partial_{i+1}x)~\text{ for }i>0;\\\label{eq twisting function 4}
    s_i\tau(x)&=\tau(s_{i+1}x)~\text{ for }i\geq 0;\\\label{eq twisting fuction 5}
    \tau(s_0(x))&=1_n.
\end{align}

Here $\partial_i$ and $s_i$ are the face and degeneracy maps, and $1_n$ is the group identity of~$G_n$ with $x\in X_n$.

Suppose $G$ acts on a simplicial set $F$ from the left. 
Given a twisting funtion $\tau\colon X\to G$, the \emph{twisted Cartesian product}  is the simplicial set $X\times_\tau F$, whose $n$-th component is
\[
X_n\times F_n,
\]
together with the face and degeneracy maps given by
\[\begin{split}
   \partial_0(x,f)&=(\partial_0x,\, \tau(x)\cdot \partial_0f);\\ 
   \partial_i(x,f)&=(\partial_ix,\,  \partial_if) ~ \text{ for } i>0;\\
   s_i(x,f)&=(s_ix,\, s_if) ~ \text{ for } i\geq 0.
\end{split}
\]

\subsection{Szczarba's twisting cochain and twisted shuffle map}\label{subsection Szczarba's twisting cochain}

We follow the exposition in~\cite[Section~5]{Franz24}; see also Szczarba’s original
paper~\cite{Szczarba1961} and the later developments~\cite{franz2021szczarba,HessTonks10}.

For $n\geq 1$, let $S_n$ be the set of $n!$ sequences of integers
\[
S_{n}=\{\bi=(i_1,\ldots, i_n)\in\NN^n\mid 0\leq i_k\leq n-k \text{ for }1\leq k\leq n\}.
\]
In particular $i_n=0$. If $n=0$, put $S_0$ to be the empty sequence $\emptyset$ as unique element. 
The sign of $\bi\in S_n$ is
\[
(-1)^{|\bi|}, \text{ where } |\bi|=i_1+\cdots+i_n.
\]
 
For $n\geq 1$, $\bi\in S_n$, $0\leq k\leq n$ and $m\geq n-k$, there are simplicial operators of degree $k$
\begin{equation}\label{eq simplicial operators D}
D_{\bi,k}\colon X_{m}\to X_{m+k}
\end{equation}
inductively defined by 
\[
D_{\emptyset,0}=\mathrm{id}\quad
\text{and}\quad
D_{\bi,k}=
\begin{cases}
    D'_{\bi',k}s_0\partial_{i_1-k} &\text{ if }k<i_1,\\
    D'_{\bi',k} &\text{ if }k=i_1,\\
    D'_{\bi',k-1}s_0 &\text{ if }k>i_1,\\
\end{cases}
\]
where $\bi'=(i_2,\ldots, i_n)$, and $D'$ is the derived simplicial map of $D$; see~\cite[p. 199]{Szczarba1961} or~\cite[p. 1863]{HessTonks10} for the definition of \emph{derived} simplicial map. 

Let $\tau \colon X\to G$ be a twisting function.
For convenience, we write $\sigma(x)=\tau(x)^{-1}$ for $x\in X_{>0}$.
For $n\geq 1$ and $\bi\in S_{n-1}$, the \emph{Szczarba operators} are the functions
\begin{equation}\label{eq Szczarba operators}
\begin{split}
{\sz}_{\bi}\colon  X_{n}&\to G_{n-1}
\\
x&\mapsto D_{\bi,0}\sigma(x)
\,D_{\bi,1}\sigma(\partial_0 x)\,\cdots\, D_{\bi,n-1}\sigma((\partial_0)^{n-1}x).
\end{split}
\end{equation}
Szczarba's \emph{twisting cochain} $t_{\sz}\colon \mathsf{C}(X)\to \mathsf{C}(G)$ is defined by
\begin{equation}\label{eq Szczarba twisting cochain}
t_{\sz}(x)=\begin{cases}
    0 & \text{ if } n=0,\\
    \sz_{\emptyset}x-1_0=\sigma(x)-1_0 &\text{ if } n=1,\\
    \sum_{\bi\in S_{n-1}}(-1)^{|\bi|}\sz_{\bi}x & \text{ if } n\geq 2,
\end{cases}
\end{equation}
where $\deg x=n$ and $1_0\in G_0$ is the group identity.

If $G$ acts on a simplicial set $F$, then $\mathsf{C}(F)$ turns into a $\mathsf{C}(G)$-module, with action map given by the composite
\[
\mathsf{C}(G)\otimes \C(F)\xrightarrow{\nabla} \C(G\times F)\xrightarrow{\mu_*} \C(F),
\]
where $\nabla$ is the shuffle map, and $\mu_*$ is induced by the $G$-action on $F$.
With the twisting cochain~\eqref{eq Szczarba twisting cochain}, we obtain a twisted tensor product, i.e., a chain complex \[(\C(X)\otimes_{t_{\sz}} \C(F), \partial_{t_{\sz}}).\]


For $n\geq 0$ and $\bi\in S_n$, define
\[\begin{split}
\hatsz_{\bi}\colon  X_{n}&\to (X\times_\tau G)_n=X_n\times G_n
\\
x&\mapsto \left(D_{\bi,0}x,
\,D_{\bi,1}\sigma(x)\,D_{\bi,2}\sigma(\partial_0 x)\,\cdots\, D_{\bi,n}\sigma\big((\partial_0)^{n-1}x\big)\right).
\end{split}
\]
When $\bi=\emptyset$ and $n=0$, we have $\hatsz_\emptyset (x)=(x,1)\in X_0\times G_0$. 
The \emph{Szczarba's twisted shuffle map} is defined as
\begin{equation}\label{eq Szczarba twisted shuffle map}
\begin{split}
\psi\colon \mathsf{C}(X)\otimes_{t_{\sz}}\mathsf{C}(F) & \to \mathsf{C}(X\times_\tau F)\\
x\otimes y&\mapsto \sum_{\bi\in S_n}(-1)^{|\bi|}(\mathrm{id}_X,\mu)_*\nabla\!\left(\hatsz_{\bi}x\otimes y\right),
\end{split}
\end{equation}
where $n$ is the degree of~$x$, and $\nabla\colon \mathsf{C}(X\times_\tau G)\otimes \mathsf{C}(F)\to \mathsf{C}(X\times_\tau G\times F)$ is the shuffle map, and $\mu\colon G\times F\to F$ is the group action. 

\begin{theorem}[{\cite[Theorem 2.4]{Szczarba1961}}]\label{thm Szczarba twisted shuffle}
The map $\psi$ is a quasi-isomorphism.
In particular it induces an isomorphism
\[
H_*(\mathsf{C}(X)\otimes_{t_{\sz}}\mathsf{C}(F))\cong H_*(X\times_\tau F).
\]
\end{theorem}


\begin{bigremark}
It is proved in~\cite[Corollary~B.3]{Franz24} that Szczarba's twisting cochain and the twisted shuffle map~\eqref{eq Szczarba twisted shuffle map} descend to normalised chain complexes
\begin{equation}
t_{\sz}\colon \N(X) \to \N(G)\quad\text{and}\quad
\N(X)\otimes_{t_{\sz}}\N(F)\to \N(C\times_\tau F).
\end{equation}
\end{bigremark}

\subsection{Marked simplicial sets and path (co)homology}

A \emph{marked simplicial set} is a pair $(X,M)$, where~$X$ is a simplicial set and $M$ is a subset of the $1$-simplices that contains all degenerate $1$-simplices. Elements of $M$ are called \emph{marked edges}.
A \emph{morphism} of marked simplicial sets is a simplicial map that sends marked edges to marked edges.
The category of marked simplicial sets will be denoted by~$\mathsf{sSet}^+$.

\begin{example}
There are different ways to regard a simplicial set as a marked simplicial set. For instance, one may take all $1$-simplices that are marked. The resulting marked simplicial set is denoted by $X^{\musSharp{}}=(X,X_1)$, following the notation of~\cite[p. 147]{Lurie}.
This gives rise to a functor $(-)^{\musSharp{}}\colon \mathsf{sSet}\xrightarrow{} \mathsf{sSet}^+$ from the category of simplicial sets to the category of marked simplicial sets.  
\end{example}

Given a marked simplicial set $(X,M)$, an element $x\in X_n$ for $n\geq 0$ is called \emph{allowed} if it satisfies
\[x_{k-1,k}\colonequals \partial_0\ldots \partial_{k-2}{\partial}_{k+1}\ldots \partial_n(x)\in M\quad\text{for all $1\leq k\leq n$}.\]
Denote by $\mathcal{A}_n(X,M)$ the set of all allowed $n$-simplices and define
\begin{equation}\label{eq defn A_n}
\A_n(X,M)=\rho\left(\kk\cdot \mathcal{A}_n(X,M)\right)
\end{equation}
which is a submodule of $\N_n(X)$. 

\begin{definition} 
Denote by $\A_*(X,M)=\bigoplus_{n\geq 0}\A_n(X,M)$.
The path chain complex of~$(X,M)$, denoted by $\Omega_\bullet(X,M)$, is the maximal subcomplex contained in~$\A_*(X,M)$. Its $n$-th component~is given by
\[
\Omega_n(X,M)\colonequals \{x\in \A_n(X,M)\mid \partial x\in \A_{n-1}(X,M)\}.
\]
The homology of $\left(\Omega_\bullet(X,M),\partial\right)$ is called the \emph{path homology} of  $(X,M)$, denoted by~$\PH_*(X,M)$.  We sometimes write $\PH_*(X,M;\kk)$ to emphasize the coefficient ring~$\kk$. 
\end{definition}

\smallskip

Dually, the cochain complex $\mathsf{C}^\bullet(X)$ is the dual of $\mathsf{C}(X)$ with $n$-th component
\[
\mathsf{C}^n(X)=\mathrm{Hom}_{\kk}(\C_n(X),\kk).
\]
Similarly, one has cochain complexes $\mathsf{D}^\bullet(X)$ and $\mathsf{N}^\bullet(X)$, which are the dual complexes of $\mathsf{D}(X)$ and $\N(X)$.
Furthermore, the diagonal map $X\xrightarrow{} X\times X$ induces a cup product on $\mathsf{N}^\bullet(X)$
\[
\cup \colon \mathsf{N}^\bullet(X) \otimes \mathsf{N}^\bullet(X) \xrightarrow{} \mathsf{N}^\bullet(X)
\]
which turns $\mathsf{N}^\bullet(X)$ into a differential graded algebra.

For a marked simplicial set $(X,M)$, let us write~$\A^n(X,M)=\hom_{\kk}(\A_n(X,M),\kk)$ the dual of $\A_n(X,M)$ over $\kk$.
The kernel
\[
\bigoplus_{n\geq 0}\N^n(X) \to  \bigoplus_{n\geq 0}\A^n(X,M)
\]
is a two-sided ideal of $\N^\bullet(X)$ with respect to the cup product $\cup$, which may not necessarily be a dg-ideal.
Let $\mathcal{J}(X,M)$ be the smallest dg-ideal containing this kernel. The $n$-th component of $\mathcal{J}(X,M)$ can be expressed as
\[
\ker (\N^n (X) \to  \A^n(X,M))+\partial^\vee \big( \ker (\N^{n-1} (X) \to  \A^{n-1}(X,M)) \big),
\]
where $\partial^\vee$ is the dual of $\partial$.

\begin{definition}
The path cochain algebra of $(X,M)$ is defined as the quotient algebra 
\[
\Omega^\bullet(X,M)\colonequals \frac{\mathsf{N}^\bullet(X)}{\mathcal{J}(X,M)}.
\]
Its cohomology is referred to as the \emph{path cohomology} of $(X,M)$, denoted by $\PH^*(X,M)$.
\end{definition}

In general $\Omega^\bullet(X,M)$ need not be isomorphic to the dual of $\Omega_\bullet(X,M)$. 
For instance, when $\kk=\mathbb{Z}$, there are examples in which $\Omega^\bullet(X,M)$ contains torsion even though $\Omega_\bullet(X,M)$ is torsion-free; see~\cite[Theorem 5.4]{fu2024path}.
However, an isomorphism does hold if $\kk$ is a field. A proof of this fact appears in~\cite[Proposition 3.16]{ivanov2024simplicial}; for reader's convenience we include a proof below.

\begin{proposition}\label{prop dual}
If $\kk=\FF$ is a field, there is an isomorphism of cochain complexes 
\[
\Omega_\bullet(X,M)^\vee\cong \Omega^\bullet(X,M),
\]
where $\Omega_\bullet(X,M)^\vee=\hom_{\FF}(\Omega_\bullet(X,M),\FF)$ is equipped with the dual differential. 
\end{proposition}
\begin{proof}
For any $n\geq 0$, the definition of~$\Omega_n(X,M)$ gives an exact sequence
\begin{equation}\label{eq exact sequence}
0\to \Omega_n(X,M)\xrightarrow{g} \N_n(X)\xrightarrow{f} \frac{\N_n(X)}{A_n(X,M)}\oplus \frac{\N_{n-1}(X)}{A_{n-1}(X,M)}
\end{equation}
where $f(a)=\big(a+A_n(X,M), \partial(a)+A_{n-1}(X,M)\big)$ and $g$ is the inclusion. 

Let $K_n=\ker \big(\N^n(X) \to A^n(X,M)\big)$. Since $\FF$ is a field, there is an isomorphism 
\begin{equation}\label{eq def Kn}
K_n\cong \left(\frac{\N_n(X)}{A_n(X,M)}\right)^\vee.
\end{equation}
Dualising the exact sequence~\eqref{eq exact sequence} over $\FF$, we have an exact sequence
\[
\left(\frac{\N_n(X)}{A_n(X,M)}\right)^\vee\oplus \left(\frac{\N_{n-1}(X)}{A_{n-1}(X,M)}\right)^\vee\xrightarrow{f^\vee} \N^n(X)\xrightarrow{g^\vee} \Omega_n(X,M)^\vee\to 0.
\]
By definition of $\mathcal{J}(X,M)$ and the identification~\eqref{eq def Kn}, the image of $f^\vee$ is exactly equal to~$\mathcal{J}(X,M)$. 
The exactness of the above sequence means $\ker g^\vee=\mathcal{J}(X,M)$, so $g^\vee$ induces an isomorphism
\[
\Omega^\bullet(X,M)\to \Omega_\bullet(X,M)^\vee.
\]
Moreover $\ker g^\vee=\mathcal{J}(X,M)$ is a dg-ideal and $g^\vee$ is a cochain map. Hence the above isomorphism is an isomorphism of cochain complexes. 
\end{proof}

\subsection{Products of marked simplicial sets}

Let $(X,M_X)$ and $(Y,M_Y)$ be two marked simplicial sets.
The \emph{strong product} 
\[(X,M_X)\times (Y,M_Y)=(X\times Y, M_X\times M_Y)\]
is the marked simplicial set whose underlying simplicial set is the product $X\times Y$, and whose marked edges are pairs of marked edges from $X$ and $Y$.  
The \emph{box product} 
\[(X,M_X)\Box (Y,M_Y)=(X\times Y, M_X\Box M_Y)\]
has the same underlying simplicial set $X\times Y$, but its set of marked edges is
defined by
\begin{multline*}
M_X\Box M_Y
=
M_X\times s_0(Y_0)\;\cup\; s_0(X_0)\times M_Y \\
=
\{(m,s_0(y))\mid m\in M_X,\ y\in Y_0\}
\;\cup\;
\{(s_0(x),m')\mid x\in X_0,\ m'\in M_Y\}.
\end{multline*}
Equivalently, an edge in $X\times Y$ is marked in the box product if and only if
one of its components is a degenerate $1$-simplex and the other component is
marked.

\begin{theorem}[{\cite[Theorem~7.6]{GLMY}, \cite[Theorem~8.6]{ivanov2024simplicial}}]\label{theorem box product Kunneth formula}
Let $\kk$ be a principal ideal domain.
The Eilenberg--Zilber map and the Alexander--Whitney map
\[
\begin{aligned}
EZ\colon&\ \N(X)\otimes \N(Y)\to \N(X\times Y),\\
AW\colon&\ \N(X\times Y)\to \N(X)\otimes \N(Y),
\end{aligned}
\]
induce mutually inverse isomorphisms of chain complexes
\[
\Omega_\bullet(X,M_X)\otimes \Omega_\bullet(Y,M_Y)\cong \Omega_\bullet\bigl((X,M_X)\Box (Y,M_Y)\bigr).
\]
    
\end{theorem}





\section{Marked twisted Cartesian products}

Let $\simgp$ be a simplicial group. We study marked simplicial sets equipped with a~$\simgp$-action. 
\begin{definition}\label{def simplicial action on a marked simplicial complex}
A simplicial group $\simgp$ is said to act from the left on a marked simplicial set $(F,M)$ if
\begin{itemize}
    \item $\simgp$ acts on the simplicial set $F$, that is, there is a simplicial map
    \[
    \Phi\colon \simgp\times F\xrightarrow{} F
    \]
    satisfying the usual associativity and identity conditions;
    \item the action preserves marked edges; in other words, $M$ is $\simgp_1$-invariant: for every $m\in M\subset F_1$ and $g\in\simgp_1$, one has $\Phi(g,m)\in M$.
\end{itemize} 
We will write $gx$ for $\Phi(g,x)$. 
\end{definition}

When $(F,M)$ is a marked simplicial set equipped with a left $G$-action, 
we may consider the twisted Cartesian product $X\times_\tau F$, equipped with a specific marking as the box product $X^{\musSharp{}}\Box (F,M)$. We refer to it as \emph{marked twisted Cartesian product}.

\begin{definition}\label{def marked twisted Cartesian product}
Let $G$ act on a marked simplicial set $(F,M)$ from the left,
and let $\tau\colon X_{>0}\to G$ be a twisting function.
 The \emph{marked twisted Cartesian product} \[ X^{\musSharp{}}\Box_\tau (F,M)\] is defined as the twisted Cartesion product
\[X\times_\tau F\]
with the marked edges given by
\[
s_0(X_0)\times M\cup X_1\times s_0(F_0).
\] 
\end{definition}
We will abbreviate the notation $X^{\musSharp{}}\Box_\tau (F,M)$ as $X\Box_\tau F$ when there is no confusion with the unmarked case $X\times_\tau F$.

\begin{bigremark}\label{remark marked twisted Cartesian product and box product}
Similar to the distinction between $X\times_\tau F$ and $X\times F$,
the marked twisted Cartesian product $X\Box_\tau F$ differs from the box product $X^{\musSharp{}}\Box (F,M)$ only in the zero-th face map.
\end{bigremark}

\begin{lemma}\label{lemma actions}
Suppose $G$ acts on the marked simplicial set $(F,M)$ from the left.
\begin{enumerate}
\item 
The action map
$\Phi\colon G\times F\to F$
induces a morphism of marked simplicial sets
\[
\widetilde{\Phi}\colon G^{\musSharp{}}\,\Box\, (F,M) \to (F,M).
\]
\item The induced chain-level action map
\[
\N(G)\otimes \N(F)\xrightarrow{\nabla} \N(G\times F)\xrightarrow{\Phi_*} \N(F)
\]
restricts to a map on $\N(G)\otimes\Omega_\bullet(F,M)$, yielding an action map
\[
\mu_{(F,M)}\colon \N(G)\otimes \Omega_\bullet(F,M)\to \Omega_\bullet(F,M).
\]
In particular,
the chain complex $\Omega_\bullet(F,M)$ becomes a left $\N(G)$-module. 
\end{enumerate}
\end{lemma}
\begin{proof}
Part (1) is a direct consequence of Definition~\ref{def simplicial action on a marked simplicial complex}.
For part~(2), Theorem~\ref{theorem box product Kunneth formula} shows that the shuffle map
\[
\nabla\colon \N(G)\otimes \N(F)\to \N(G\times F)
\]
restricts to an isomorphism
\[
\N(G)\otimes \Omega_\bullet(F,M)=\Omega_\bullet(G^{\musSharp{}})\otimes \Omega_\bullet(F,M) \to \Omega_\bullet(G^{\musSharp{}}\Box (F,M)).
\] 
The desired action map $\mu_{(F,M)}$ is then obtained by composing this isomorphism with the induced map
$\widetilde{\Phi}_*\colon \Omega_\bullet(G^{\musSharp{}}\Box (F,M))\to \Omega_\bullet(F,M)$.
\end{proof}


\begin{lemma}\label{lemma naturality wrt group action}
Suppose that for any $g\in G_1$ and $m\in F_0$, the simplex $gs_0(m)\in F_1$ is degenerate.
Then the simplicial map
\[
(\id_X,\Phi)\colon X\times_\tau G \times F\to X\times_\tau F
\]
gives a morphism of marked simplicial sets
\[
\widetilde{(\id_X,\Phi)}\colon 
(X\times_\tau G)^{\musSharp{}}\Box (F,M)\to 
X\Box_\tau F.
\]
\end{lemma}
\begin{proof} 
The marked edges of the box product $(X\times_\tau G)^{\musSharp{}}\Box (F,M)$ are of two types:
\[\begin{split}
(s_0x,s_0g, m) &\, \text{ for $(x,g)\in X_{0}\times G_{0}$ and $m\in M$},\\
(x,g,s_0(m)) &\, \text{ for $(x,g)\in X_1\times G_{1}$ and $m\in F_0$}.
\end{split}
\]
In the first case, we compute
$(\id_X,\Phi)(s_0x,s_0g, m)=(s_0x,(s_0g) m)\in s_0(X_0)\times M$,
since $m\in M$ and the action preserves markings.
In the second case,
we have
$(\id_X,\Phi)(x,g, s_0(m))=(x,g s_0(m))\in X_1\times s_0(F_0)$,
which is marked, because $g s_0(m)$ is degenerate by assumption. 

In both cases,  the image of a marked edge under $(\id_X,\Phi)$ is again a marked edge. 
Hence $(\id_X,\Phi)$ induces a morphism of marked simplicial sets, as claimed.
\end{proof}

\section{Szczarba's twisted shuffle map for marked twisted Cartesian products}
\label{section marked twisted Cartesian products}

Consider Szczarba's twisting cochain on normalised chain complexes
\[
t_{\sz}\colon \N(X)\to \N(G).
\]
Together with the action map $\mu_{(F,M)}$ from Lemma~\ref{lemma actions}, this yields a chain complex
\[
\big(\N(X)\otimes_{t_{\sz}} \Omega_\bullet(F,M), \partial_{t_{\sz}}\big),
\]
where $\partial_{t_{\sz}}=\partial_X\otimes 1+1\otimes \partial_F- (1\otimes \mu_{(F,M)}) (1\otimes t_{\sz}\otimes 1)(\Delta_X\otimes 1)$.

Our goal of this section is to prove that Szczarba's twisted shuffle map $\psi$ on normalised chain complexes restricts to a chain map
\begin{equation}\label{eq unclear chain map}
\big(\N(X)\otimes_{t_{\sz}}\Omega_\bullet(F,M),\partial_{t_{\sz}}\big)\to (\Omega_\bullet(X\times_\tau F),\partial)
\end{equation} 
which induces an isomorphism on homology.

\begin{lemma}\label{lemma Szczarba's twisted shuffle map}
Suppose that for any $g\in G_1$ and $m\in F_0$, the simplex $gs_0(m)\in F_1$ is degenerate.
Then Szczarba's twisted shuffle map $\psi$ on normalised chain complexes can be restricted to a map of chain complexes
\[
\psi_\omega\colonequals \left.\psi\right|_{\N(X)\otimes_{t_{\sz}} \Omega_\bullet(F,M)} \colon \big(\N(X)\otimes_{t_{\sz}} \Omega_\bullet(F,M),\partial_{t_{\sz}}\big)\to (\Omega_\bullet(X\Box_\tau F),\partial).
\]
\end{lemma}
\begin{proof}
By~Theorem~\ref{theorem box product Kunneth formula}, the (untwisted) shuffle map on normalised chain complexes
\[
 \mathsf{N}(X\times_\tau G)\otimes \mathsf{N}(F)\to \mathsf{N}(X\times_\tau G\times F)
\]
restricts to an isomorphism
\begin{multline*}
\nabla \colon \mathsf{N}(X\times_\tau G)\otimes \Omega_\bullet(F,M)=\Omega_\bullet((X\times_\tau G)^{\musSharp{}})\otimes \Omega_\bullet(F,M)
 \to
\Omega_\bullet\!\left((X\times_\tau G)^{\musSharp{}}\Box (F,M)\right). 
\end{multline*}

Now take $x\otimes y\in \N_n(X)\otimes_{t_{\sz}} \Omega_\bullet(F,M)$. For each $\bi\in S_n$, as
$\hatsz_{\bi}x\in \N_n(X\times_\tau G)$,
we apply $\nabla$ to $\hatsz_{\bi}x\otimes y$ to obtain
\[\nabla\!\left(\hatsz_{\bi}x\otimes y\right)\in \Omega_\bullet\!\left((X\times_\tau G)^{\musSharp{}}\Box (F,M)\right).\]
By the naturality of~$\widetilde{(\mathrm{id}_X,\Phi)}$, established in Lemma~\ref{lemma naturality wrt group action}, we have
\begin{equation}\label{eq image in Omega}
\widetilde{(\mathrm{id}_X,\Phi)}_*\nabla\!\left(\hatsz_{\bi}x\otimes y\right)\in \Omega_\bullet(X\Box_\tau F).
\end{equation}
According to~\eqref{eq Szczarba twisted shuffle map},
the twisted shuffle map is given by
\begin{equation}\label{eq chain map for marked twisted tensor products}
\psi(x\otimes y)
=\sum_{\bi\in S_n}(-1)^{|\bi|}\widetilde{(\mathrm{id}_X,\Phi)}_*\nabla\!\left(\hatsz_{\bi}x\otimes y\right).
\end{equation}
Thus Equations~\eqref{eq image in Omega} and \eqref{eq chain map for marked twisted tensor products} imply \( \psi(x \otimes y) \in \Omega_\bullet(X \Box_\tau F)\). 
Since $\psi$ is a chain map, its restriction to subcomplexes~\eqref{eq chain map for marked twisted tensor products} defines a chain map
\[
\psi_\omega \colon \N(X)\otimes_{t_{\sz}} \Omega_\bullet(F,M) \longrightarrow \Omega_\bullet(X\Box_\tau F),
\]
as claimed.
\end{proof}

To show $\psi_\omega$ is an isomorphism, we first provide a statement analogue to~\cite[Lemma 5.3]{Szczarba1961}
in the case of marked twisted Cartesian products.
For this, let us consider the shuffle map
\[
\nabla_{X,F}\colon \N(X)\otimes_{t_{\sz}}\N(F)\to \N(X\times_\tau F)
\]
and the Alexander-Whitney map
\[
AW_{X,F}\colon N(X\times_\tau F)\to \N(X)\otimes_{t_{\sz}}\N(F).
\] 
By Remark~\ref{remark marked twisted Cartesian product and box product} and~Theorem~\ref{theorem box product Kunneth formula}, they induce the following isomorphisms
\[
\nabla_{X,F}\colon \N(X)\otimes_{t_{\sz}} \Omega_\bullet(F,M)=\Omega_\bullet(X^{\musSharp{}})\otimes_{t_{\sz}} \Omega_\bullet(F,M)\rightleftarrows \Omega_\bullet(X\times_\tau F) \cocolon AW_{X,F}
\]
which are inverse of each other.

Note that $\nabla_{X,F}$ and $AW_{X,F}$ are not chain maps in general, but they preserve the following filtrations 
\begin{equation}
\label{eq Ap Bp}
\begin{split}
A_p&=\bigoplus_{q\leq p}\N_q(X)\otimes_{t_{\sz}} \Omega_\bullet(F,M),\\
B_p&=\Omega_\bullet(\sk_p(X)\Box_\tau F).
\end{split} \end{equation}
Here $\sk_p(X)$ is the $p$-th skeleton of $X$, generated by all $q$-simplices with $q\leq p$.
Therefore, we obtain maps
\begin{equation}\label{eq iso shuffle aw on quotients}
\nabla'_{X,F}\colon A_p/A_{p-1}\rightleftarrows B_p/B_{p-1} \cocolon  AW'_{X,F}.
\end{equation}
\begin{lemma}[cf. {\cite[Lemma 5.3]{Szczarba1961}}]
\label{lemma differential}
The maps $\nabla'_{X,F}$ and $AW'_{X,F}$ are chain maps such that
\[
\text{$\nabla'_{X,F}AW'_{X,F}=\text{identity}~$ and $~AW'_{X,F}\nabla'_{X,F}=\text{identity}$.}
\]
\end{lemma}
\begin{proof} 
If $x\otimes y\in A_p$, then
\begin{equation}
\begin{split}
\partial_{t_{\sz}}(x\otimes y)
&=(\partial x)\otimes y+(-1)^{p} x\otimes \partial y-\sum \pm x^{(1)}\otimes \mu_{(F,M)}(t_{\sz}(x^{(2)})\otimes y),
\end{split}
\end{equation}
where the right-hand side sum ranges over the decomposition $\Delta(x)=\sum x^{(1)}\otimes x^{(2)}$.
Notice that $t_{\sz}(x^{(2)})=0$ if $\deg x^{(2)}=0$, so 
$(\partial x)\otimes y$ and $x^{(1)}\otimes \mu_{(F,M)}(t_{\sz}(x^{(2)})\otimes y)$ lie in~$A_{p-1}$.
Therefore, we have 
\begin{equation}\label{eq differential on A_p/A_{p-1}}
\partial_{t_{\sz}}(x\otimes y)=(-1)^{p} x\otimes \partial y \mod A_{p-1}.    
\end{equation}
  
If $x\in X_p$ is non-degenerate and $s_\mu=s_{i_1}\ldots s_{i_r}$ for some $i_1>\cdots>i_r$, then
\begin{equation}\label{eq partial zero s mu}
\partial_0 s_\mu x=
\begin{cases}
s_{i_1-1}\ldots s_{i_{r-1}-1} x &\text{ if } i_r=0,\\
s_{i_1-1}\ldots s_{i_r-1}\partial_0 x &\text{ if } i_r>0.
\end{cases}
\end{equation}
Note that $\partial_0 s_\mu x\in \sk_{p-1}X$ in the second case, while in the first case, 
\begin{equation}
\label{eq tau s_mu}
\tau(s_\mu x)=s_{i_1-1}\ldots s_{i_{r-1}-1} \tau(s_0 x)=1_{p+r-1},
\end{equation}
following from~\eqref{eq twisting function 4}, \eqref{eq twisting fuction 5} and the fact that all $s_i$'s are group homomorphisms.

For $(s_\mu x,y)\in B_p$ with $y\in F_{p+r}$, Equations~\eqref{eq partial zero s mu} and \eqref{eq tau s_mu} imply
\[
\partial_0(s_\mu x,y)=(\partial_0 s_\mu x,\tau(s_\mu x)\cdot \partial_0 y)=
(\partial_0 s_\mu x, \partial_0 y)\mod B_{p-1}.
\]
This means
\[
\partial(s_\mu x,y)=\sum_{i=0}^{p+r}(-1)^i(\partial_i s_\mu x,\partial_i y) \mod B_{p-1}.
\]

Hence the differentials on $A_p/A_{p-1}$ and $B_p/B_{p-1}$ agree with the non-twisted product case.  By~Theorem~\ref{theorem box product Kunneth formula}, we obtain chain maps $\nabla'_{X,F}$ and $AW'_{X,F}$ that are inverse of each other.
\end{proof}

Since 
$\psi_\omega(A_p)\subset B_p$, the map
$\psi_\omega$ induces a well-defined map on quotients
\[\psi'\colon A_p/A_{p-1}\to B_p/B_{p-1}.\]
Below we describe $\psi'$ in an explicit form.

\begin{lemma}[cf. {\cite[Corollary 5.5]{Szczarba1961}}]
For $x\in \N_p(X)$ and $y\in\Omega_q(F,M)$, we have
\[
\psi'(x\otimes y)=\sum_{(\mu,\nu)\in \Sh(p,q)} (-1)^{\sgn(\mu,\nu)}\left(
s_{\nu}x,\, s_{\nu}[s_0\sigma(x)\cdots s_0^p\sigma(\partial_0^{p-1}x)]\cdot s_{\mu}y 
\right),
\] 
where $\Sh(p,q)$ is the set of all $(p,q)$-shuffles.
\end{lemma}

\begin{proof}
    
Using the definition of $\psi_\omega$,
we get
\begin{multline*}
\psi'(x\otimes y)
=\sum_{\bi\in S_p} \sum_{(\mu,\nu)\in \Sh(p,q)}(-1)^{|\bi|+\sgn(\mu,\nu)} 
\\
\left(s_\nu D_{\bi,0} x,\, s_{\nu}\!\left[D_{\bi,1}\sigma(x)\,D_{\bi,2}\sigma(\partial_0 x)\,\cdots\, D_{\bi,p}\sigma((\partial_0)^{p-1}x)\right]\cdot s_\mu y \right).
\end{multline*}
If $\bi\neq (0,\ldots,0)$, the operator $D_{\bi,0}$ contains a face map. Hence the term $s_\nu D_{\bi,0}x$ in the sum will be contained in $B_{p-1}$, so it vanishes in $B_p/B_{p-1}$. If $\bi= (0,\ldots,0)$,
then $D_{\bi,0}=\text{identity}$ and
$D_{\bi,k}=s_0^k,$
which implies the asserted formula.
\end{proof}

Let $C_p=\N(\sk_p(X)\times_\tau F).$
We recall from~\cite[p.\,208]{Szczarba1961} the chain isomorphism 
\[g\colon C_p/C_{p-1}\to C_p/C_{p-1}\] and its inverse 
\[h\colon C_p/C_{p-1}\to C_p/C_{p-1}.\] 
For any non-degenerate simplex $x\in X_p$, $y\in F_{p+r}$ and $s_\mu=s_{i_1}\ldots s_{i_r}$ with $i_1>\cdots>i_r$, we have
\[\begin{split}
g(s_\mu x, y)& =
\big(
s_\mu x,\,
s_\mu\big[s_0\sigma(x)\, \cdots \,(s_0)^p\sigma((\partial_0)^{p-1}x)\big] \cdot y
\big);
\\
h(s_\mu x, y)& =
\big(
s_\mu x,\,
s_\mu\big[(s_0)^r\tau((\partial_0)^{r-1} x)\, \cdots\, s_0\tau(x)\big] \cdot y
\big).\end{split}
\]

Since $B_{p-1}=B_p \cap C_{p-1}$, there is an injection
\[
B_p/B_{p-1}=B_p/(B_p \cap C_{p-1})= (B_p+C_{p-1})/C_{p-1}\hookrightarrow C_p/C_{p-1}.
\]
Composing this injection with $g$ and $h$ yields maps 
\[\hat{g},\, \hat{h}\colon B_p/B_{p-1}\to C_p/C_{p-1}.\]
\begin{lemma}
The images of $\hat{g}$ and $\hat{h}$ are contained in $B_p/B_{p-1}$. 
As a result, the maps $\hat{g}$ and $\hat{h}$ restrict to
\[
\hat{g}, \hat{h}\colon B_p/B_{p-1}\to B_p/B_{p-1},
\]
which are inverses of each other.   
\end{lemma}
\begin{proof}

Recall from~\eqref{eq iso shuffle aw on quotients} that there is an isomorphism
\[
\nabla'_{X,F}\colon A_p/A_{p-1}\to B_p/B_{p-1}.
\]
Thus, it is enought to show that $g\nabla'_{X,F}(A_p/A_{p-1})\in B_p/B_{p-1}$.

For any non-degenerate simplex
${x}\in X_p$ and $\yy \in \Omega_q(F,M)$, 
let us compute 
\begin{multline}\label{eq g nabla}
g\nabla'_{X,F}(\xx\otimes \yy)
=\sum_{(\mu,\nu)\in \Sh(p,q)}(-1)^{\sgn(\mu,\nu)} g (s_{\nu} \xx, s_{\mu}\yy)\\
=\sum_{(\mu,\nu)\in \Sh(p,q)}(-1)^{\sgn(\mu,\nu)}
\underbrace{\big(
s_\nu \xx,\,
s_\nu\big[s_0\sigma(\xx)\, \cdots \,(s_0)^p\sigma((\partial_0)^{p-1}\xx)\big] \cdot (s_\mu y)
\big)}_{\colonequals g_{\mu,\nu}(x\otimes y)}.
\end{multline}
This coincides with $\psi'(\xx\otimes \yy)$, which is in $B_p/B_{p-1}$.
Therefore $g$ restricts to a map
\[
\hat{g}\colon B_p/B_{p-1}\to B_p/B_{p-1}.
\]

For $h$, we also consider its composition with $\nabla'_{X,F}$:
\[
h\nabla'_{X,F}(\xx\otimes \yy)
=\sum_{(\mu,\nu)\in \Sh(p,q)}(-1)^{\sgn(\mu,\nu)}
\underbrace{\Big(
s_\nu \xx,\,
s_\nu\big[(s_0)^r\tau((\partial_0)^{r-1} x)\, \cdots\, s_0\tau(x)\big] \cdot (s_\mu y)
\Big)}_{\colonequals h_{\mu,\nu}(x\otimes y)}.
\]
The terms appearing in $g\nabla'_{X,F}(\xx\otimes \yy)$ and $h\nabla'_{X,F}(\xx\otimes \yy)$ are both indexed by~$(\mu,\nu)$. 
We denote them by $g_{\mu,\nu}(x\otimes y)$ and $h_{\mu,\nu}(x\otimes y)$, respectively. Since the $G_1$-action preserves 1-degeneracies and the marked edges, it follows by inspection that $g_{\mu,\nu}(x\otimes y)$ is allowed if and only if $h_{\mu,\nu}(x\otimes y)$ is allowed. 

Moreover, since $x$ is non-degenerate, the Eilenberg-Zilber Lemma implies that if~$s_\nu x=s_{\nu'} x$, then $s_\nu=s_{\nu'}$. 
Thus for distinct $({\mu,\nu})$, the terms $g_{\mu,\nu}$ are distinct. 
As the sum $\sum_{(\mu,\nu)}(-1)^{\sgn(\mu,\nu)}g_{\mu,\nu}(x\otimes y)=g\nabla'_{X,F}(\xx\otimes \yy)$ is allowed, it follows that every individual term $g_{\mu,\nu}(x\otimes y)$ is allowed. Consequently, each corresponding term $h_{\mu,\nu}(x\otimes y)$ is allowed as well.
Therefore the entire sum $h\nabla'_{X,F}(\xx\otimes \yy)$ is allowed.

It remains to show that
$\partial (h\nabla'_{X,F}(\xx\otimes \yy))$ is allowed. 
As shown in Lemma~\ref{lemma differential}, the differential on $B_p/B_{p-1}$ is the same as the untwisted case. Therefore, we have
\begin{align*}
\partial (g\nabla'_{X,F}(\xx\otimes \yy))
=&\sum_{(\mu,\nu)\in \Sh(p,q)}
\sum_{i=0}^{p+q}(-1)^{\sgn(\mu,\nu)}\partial_ig_{\mu,\nu}(x\otimes y),\\
\partial (h\nabla'_{X,F}(\xx\otimes \yy))
=&\sum_{(\mu,\nu)\in \Sh(p,q)}
\sum_{i=0}^{p+q}(-1)^{\sgn(\mu,\nu)}\partial_ih_{\mu,\nu}(x\otimes y),\\
\intertext{where} 
\partial_ig_{\mu,\nu}(x\otimes y)=(-1)^{i}\Big(
\partial_i s_\nu \xx,\,
\partial_i& s_\nu\big[s_0\sigma(\xx)\, \cdots \,(s_0)^p\sigma((\partial_0)^{p-1}\xx)\big] \cdot  (\partial_i s_\mu y
)\Big), \\
\partial_ih_{\mu,\nu}(x\otimes y)=(-1)^{i}\Big(
\partial_i s_\nu \xx,\,
\partial_i &s_\nu\big[(s_0)^r\tau((\partial_0)^{r-1} x)\, \cdots\, s_0\tau(x)\big] \cdot (\partial_i s_\mu y
)\Big).
\end{align*}

Again, under the assumption that the $G_1$-action preserves the degenerate $1$-simplices, we have that
$\partial_ih_{\mu,\nu}(x\otimes y)$ is allowed if and only if $\partial_ig_{\mu,\nu}(x\otimes y)$ is allowed.
Let 
\[\begin{split}
I&=\{(\mu,\nu,i)\mid \partial_ig_{\mu,\nu}(x\otimes y) \text{ is not allowed}\}\\
&=\{(\mu,\nu,i)\mid \partial_ih_{\mu,\nu}(x\otimes y) \text{ is not allowed}\}.
\end{split}\]
We define an equivalence relation $\sim$ on $I$ by
\[(\mu,\nu,i)\sim (\mu',\nu',i')\]
if and only if the norm forms of $\partial_{i}s_{\nu}$ and 
$\partial_{i'}s_{\nu'}$ are equal. 
Let $[\mu,\nu,i]\in I/\sim$ denote an equivalence class, and define 
\[
g_{[\mu,\nu,i]}\colonequals\sum_{(\mu,\nu,i)\in [\mu,\nu,i]}(-1)^{\sgn(\mu,\nu)}\partial_i g_{\mu,\nu}(x\otimes y).
\]
If $(\mu_0,\nu_0,i_0)\notin [\mu,\nu,i]$, then
$\partial_{i_0}s_{\nu_0}\xx\neq \partial_is_{\nu_i}\xx$. 
So the terms in other equivalence classes cannot cancel $g_{[\mu,\nu,i]}$.
Hence, as $g_{[\mu,\nu,i]}$ is not allowed, it must be zero 
(otherwise $\partial g\nabla'_X(\xx\otimes \yy)$ would not be allowed, a contradiction).

Moreover, each $g_{[\mu,\nu,i]}$ can be written as
\begin{multline*}
(\partial_i s_{\nu}\xx,\, \partial_is_\nu\big[s_0\sigma(\xx)\, \cdots \,(s_0)^p\sigma((\partial_0)^{p-1}\xx)\big] \cdot y^+_{[\mu,\nu,i]})-\\(\partial_i s_{\nu}\xx,\, \partial_is_\nu\big[s_0\sigma(\xx)\, \cdots \,(s_0)^p\sigma((\partial_0)^{p-1}\xx)\big] \cdot y^-_{[\mu,\nu,i]}),
\end{multline*}
where
\[y^+_{[\mu,\nu,i]}=\sum_{\substack{(\mu,\nu,i)\in [\mu,\nu,i]\\ (-1)^{\sgn(\mu,\nu)+i}=1}}
 \partial_i s_\mu y\quad
\text{and}\quad 
y^-_{[\mu,\nu,i]}=\sum_{\substack{(\mu,\nu,i)\in [\mu,\nu,i]\\ (-1)^{\sgn(\mu,\nu)+i}=-1}}
 \partial_i s_\mu y.\]
Since $g_{[\mu,\nu,i]}=0$,  we conclude
$y^+_{[\mu,\nu,i]}-y^-_{[\mu,\nu,i]}=0$. Therefore, 
\[
h_{[\mu,\nu,i]}=
\Big(
\partial_i s_\nu \xx,\,
\partial_i s_\nu\big[(s_0)^r\tau((\partial_0)^{r-1} \xx)\, \cdots\, s_0\tau(\xx)\big] \cdot 
\big(y^+_{[\mu,\nu,i]}-y^-_{[\mu,\nu,i]}\big)\Big)=0.
\]

Finally,
the total contribution from non-allowed terms in $\partial (h\nabla'_{X,F}(\xx\otimes \yy))$ is given by
\[\sum_{(\mu,\nu,i)\in I} (-1)^{\sgn(\mu,\nu)}\partial_ih_{\mu,\nu}(x\otimes y)=
\sum_{[\mu,\nu,i]\in I/\sim} h_{[\mu,\nu,i]}=0.
\]
Therefore $\partial (h\nabla'_{X,F}(\xx\otimes \yy))$ is allowed and hence
\[ h\nabla'_{X,F}(\xx\otimes \yy)\in B_p/B_{p-1}.\]

The claim that $\hat{g}$ and $\hat{h}$ are inverses of each other follows immediately from the fact that ${g}$ and ${h}$ are inverses of each other.
\end{proof}

\begin{theorem}\label{theorem Szczarba's twisted shuffle map}
Under the assumptions of Lemma~\ref{lemma Szczarba's twisted shuffle map},
the restriction of Szczarba's twisted shuffle map 
\[
\psi_\omega\colonequals \left.\psi\right|_{\N(X)\otimes_{t_{\sz}} \Omega_\bullet(F,M)} \colon (\N(X)\otimes_{t_{\sz}} \Omega_\bullet(F,M),\partial_{t_{\sz}})\to (\Omega_\bullet(X\Box_\tau F),\partial)
\]
is an isomorphism of chain complexes.
In particular, it induces an isomorphism on homology. 
Moreover, the construction is natural with respect to \(G\)-invariant marked simplicial
subsets. In other words, if \((F',M')\subset (F,M)\) is a \(G\)-invariant marked simplicial subset,
then the inclusion induces a commutative diagram of chain complexes
\[
\begin{tikzcd}
\N(X)\otimes_{t_{\sz}} \Omega_\bullet(F',M)
\arrow[r, "\psi_\omega"]
\arrow[d, hook]
&
\Omega_\bullet(X\Box_\tau F')
\arrow[d, hook]
\\
\N(X)\otimes_{t_{\sz}} \Omega_\bullet(F,M)
\arrow[r, "\psi_\omega"]
&
\Omega_\bullet(X\Box_\tau F),
\end{tikzcd}
\]
and hence the induced isomorphism is natural with respect to such inclusions.
\end{theorem}
\begin{proof}
Recall that
$A_p=\bigoplus_{q\leq p}\N_q(X)\otimes_{t_{\sz}} \Omega_\bullet(F,M)$ and $
B_p=\Omega_\bullet(\sk_p(X)\Box_\tau F)$ are increasing filtrations of chain complexes.
Since $\psi'=\hat{g}\nabla'_{X,F}$ by~\eqref{eq g nabla}, and both $\hat{g}$ and $\nabla'_{X,F}$ are isomorphisms, it follows $\psi'\colon A_p/A_{p-1}\to B_p/B_{p-1}$ is an isomorphism.

We now prove by induction on $p$ that the restriction of $\psi_\omega$
\[\psi_\omega\colon A_p\to B_p\] is an isomorphism.
For the case $p=0$, observe that $A_0$ is generated by $x\otimes f$, where $x\in X_0$ and $ f\in \Omega_\bullet(F,M)$,
while $B_0$ is generated by pairs $(s_0^r x, f)$, where $x\in X_0$ and $f\in \Omega_r(F,M)$.
Then $\psi_\omega$ maps $(x,f)$ to $(s_0^r x, f)$ which is clearly an isomorphism.

Assume $\psi_\omega\colon A_{p-1}\to B_{p-1}$ is an isomorphism.
Consider the following diagram of short exact sequences 
\[
\begin{tikzcd}
0\ar[r] &A_{p-1}\ar[r]\ar[d,"\psi_\omega"] &A_p \ar[r]\ar[d,"\psi_\omega"] & A_p/A_{p-1}\ar[d,"\psi'"]\ar[r]&0\\
0\ar[r] &B_{p-1}\ar[r] &B_p \ar[r] & B_p/B_{p-1}\ar[r]&0.
\end{tikzcd}
\]
By the inductive hypothesis, the left vertical map is an isomorphism. We have shown that the right vertical map $\psi'$
is also an isomorphism. Therefore, by the Five Lemma, the middle vertical map is an isomorphism as well.
Thus, $\psi_\omega$ is an isomorphism.
In particular, it induces an isomorphism on homology.

The naturality with respect to \(G\)-invariant marked simplicial subsets follows formally from the functoriality of Szczarba’s twisting
cochain and the definition of the twisted shuffle map $\psi$.
\end{proof}


\section{Equivariant path (co)homology}

\subsection{Borel constructions}
Let $\simgp$ be a simplicial group. Consider the universal $\simgp$-bundle $E\simgp\xrightarrow{} B\simgp$; see for instance~\cite[\S 21]{May}. 
If $\simgp$ acts on a simplicial set $X$ from the left, the Borel construction of $X$ is defined as the quotient 
\[E\simgp\times_{\simgp} X\colonequals (E\simgp\times X)/\simgp,\]
where the $\simgp$-action on  $E\simgp\times X$ is given by 
\[h(g_n,g_{n-1},\ldots, g_0,x)=(g_nh,g_{n-1},\ldots, g_0,h^{-1}x)\]
for $h\in G_n$, $g_i\in G_i$ and $x\in X_n$.
The Borel construction is also denoted by $X_{\simgp}$.

\begin{lemma}\label{lemma condition for marked borel construction}
 Let $(X,M)$ be a marked simplicial $\simgp$-set. 
 Suppose that for all $x\in X_0$ and $g\in \simgp_1$, the $1$-simplex $gs_0(x)$ is degenerate.
 Then the pair 
 \[\left(E\simgp\times_{\simgp} X, ((E\simgp)_1\Box M)/\simgp_1
\right)\]
is a marked simplicial set.
\end{lemma}
\begin{proof}
It suffices to show that the $\simgp$-action on $E\simgp\times X$ preserves the marked edges in $(E\simgp)_1\Box M$. 
Remember that marked edges of $(E\simgp)_1\Box M$ are of two types:
\[
\begin{split}
(g_1,g_0,s_0(m)) \quad&\text{for $g_i\in G_i$ and $m\in X_0$},  \\
 (s_0(g_0),1, m) \quad & \text{for $g_0\in G_0$ and $m\in M$}.
\end{split}
\]

Let $h\in \simgp_1$. In the first case,  we have 
\[
h(g_1,g_0,s_0(m))=(g_1h,g_0, h^{-1}s_0(m))).
\]
Since $h^{-1}s_0(m)$ is degenerate by assumption, this simplex is marked in $(E\simgp)_1\Box M$.

In the second case,
we compute
\[
h(s_0(g_0),1, m)=(s_0(g_0)h,1, h^{-1}m).
\]
Here $(s_0(g_0)h,1)$ is still degenerate in $EG$, and since $m\in M$ and the marked edges in $M$ is $\simgp$-invariant, we have $h^{-1}m\in M$. So again the simplex $h(s_0(g_0),1, m)$ is marked.

Therefore, the marked edges are preserved under the $\simgp$-action, which descends a marking on the quotient simplicial set.  
\end{proof}

\begin{definition}
Let $(X,M)$ be a marked simplicial $\simgp$-set such that the condition of the lemma above. The Borel construction of $(X,M)$ is the marked simplicial set
\[
(X,M)_\simgp\colonequals \left(E\simgp\times_\simgp X, ((E\simgp)_1\Box M)/\simgp_1
\right).\] 
The path homology (resp. cohomology) of $(X,M)_\simgp$ is called the \emph{equivariant path homology} (resp. \emph{cohomology}) of~$(X,M)$.
\end{definition}

On the other hand, recall that there is a twisting function
\[
\tau_G\colon (BG)_{n+1}=G_n\times\cdots\times G_0\to G_n
\]
given by the projection onto the first factor.
Suppose $G$ acts on a marked simplicial set $(X,M)$ from the left. We obtain a marked twisted Cartesian product
\[
(BG)^{\musSharp{}} \Box_{\tau_G} (X,M).
\]

\begin{proposition}\label{prop Borel construction is a twisted Cartesian product}
Under the hypothesis of Lemma~\ref{lemma condition for marked borel construction},
there exists an isomorphism of marked simplicial sets
\[
(X,M)_G \cong (BG)^{\musSharp{}} \Box_{\tau_G} (X,M).
\]
\end{proposition}
\begin{proof}
Write $[g,b,x]$ for the equivalence class of $(g,b,x)\in EG_n\times X_n$.
Define maps
\[
\varphi\colon EG\times_G X\to BG\times_{\tau_G} X \quad \varphi[g,b,x]=(b,gx)
\]
and
\[\psi\colon BG\times_{\tau_G} X\to EG\times_G X,\quad
\psi(b, x)=\big[1_n,b, x\big],
\]
where $1_n$ is the group identity of $G_n$.
A quick check show that $\varphi, \psi$ are simplicial and inverse to each other.

To see $\varphi$ and $\psi$ preserve the marked edges, we consider several cases.
If $(g,b)$ is degenerate in $EG_1$ and $x\in M$, then $b$ is degenerate in $BG_1$ and $gx\in M$. 
If $x$ is degenerate in~$X_1$, then $gx$ is also degenerate by hypothesis. Thus $\varphi$ preserves marking.
Conversely, if $x$ is degenerate, then $[1_n,b,x]$ is marked in $EG\times_G X$. If~$b$ is degenerate in $BG_1$ and $x\in M$, then $(1_1,b)$ is degenerate in $EG_1$. Hence $\psi(b,x)$ is marked. Thus $\psi$ preserves marked edges as well.
\end{proof}

\begin{corollary}\label{cor model for equivariant homology}
There is an isomorphism of chain complexes
\[
\N(BG)\otimes_{t_{\sz}} \Omega_\bullet(X,M)\to \Omega_\bullet\big((X,M)_G\big),
\]
which is natural with respect to the inclusion of \(G\)-invariant marked simplicial
subsets $(X',M')\hookrightarrow (X,M)$.
\end{corollary}
\begin{proof}
Applying Theorem~\ref{theorem Szczarba's twisted shuffle map} to the marked twisted Cartesian product \[(BG)^{\musSharp{}}\Box_{\tau_G} (X,M)\]
yields an isomorphism
\[
\N(BG)\otimes_{t_{\sz}} \Omega_\bullet(X,M)\to 
\Omega_\bullet((BG)^{\musSharp{}}\Box_{\tau_G} (X,M)).
\]
Then the assertion follows from Proposition~\ref{prop Borel construction is a twisted Cartesian product}.
\end{proof}
    
\section{Borel construction of directed graphs}

A directed graph $\mathcal{G}=(V, E)$ consists of a set of vertices $V$ and a set of directed edges $E\subset V\times V$. 
In this work, following~\cite{ivanov2024simplicial} we assume $E^D=\{(v,v)\mid v\in V\}\subset E$, which differs from the usual convention in graph theory. An edge $(v,v)$ is \emph{degenerate}. A non-degenerate edge $(u,v)\in E$ is depicted as $u\to v$,
while degenerate edges are not drawn (one may think of them as ``ghost loops'').

\subsection{Review of GLMY path chains}
Consider the free~\(\kk\)-module \(\mathcal{R}_n(V)\) generated by all \((n+1)\)-tuples of vertices such that every adjacent vertices are distinct. The basis element of \(\mathcal{R}_n(V)\) corresponding to the tuple $(v_0,v_1,\ldots, v_n)$ is denoted by $e_{v_0,v_1,\ldots, v_n}$. 
The boundary operator $\partial\colon \mathcal{R}_n(V)\to \mathcal{R}_{n-1}(V)$ is defined as
\begin{equation}\label{eq boundary operator}
\partial e_{v_0,v_1,\ldots, v_n}=\sum_{\substack{0\leq i\leq n\\{v}_{i-1}\neq v_{i+1}}}(-1)^ie_{v_0,\ldots,{v}_{i-1},v_{i+1},\ldots, v_n}.
\end{equation} 

For \(n\ge 0\), an \emph{\(n\)-path} in \(\mathcal{G}\) is an $(n+1)$-tuple
$(v_0,v_1,\ldots,v_n)$
such that 
\[\text{\((v_{i-1},v_i)\in E\) and \(v_{i-1}\neq v_i\) for all \(1\le i\le n\).}\] 
Denote by \(\mathcal{A}_n(\mathcal{G})\subset \mathcal{R}_n(V)\) the submodule generated by \(n\)-paths of~\(\mathcal{G}\). In general the operator $\partial$ in~\eqref{eq boundary operator} does not preserve $\mathcal{A}_n(\mathcal{G})$.

To obtain a chain complex, define \(\Omega_n(\mathcal{G})\subset \mathcal{A}_n(\mathcal{G})\) consisting of those elements whose boundary lies in \(\mathcal{A}_{n-1}(\mathcal{G})\), i.e.,
\[
\Omega_n(\mathcal{G})=\{x\in\mathcal{A}_n(\mathcal{G})\mid \partial x\in \mathcal{A}_{n-1}(\mathcal{G}) \}.
\]
The restriction of \(\partial\) to \(\Omega_n(\mathcal{G})\) yields a chain complex
\[
(\Omega_\bullet(\mathcal{G}),\partial).
\]
Its homology is the \emph{GLMY path homology} of \(\mathcal{G}\), denoted as~\(\PH_*(\mathcal{G})\).

The GLMY path homology admits a natural interpretation in the language of path homology of marked simplicial sets. To a directed graph \(\mathcal{G}=(V,E)\) we associate a marked simplicial set
\[
\mkd{\mathcal{G}}=(\Nrv(\mathcal{G}),E).
\]
Here $\Nrv(\G)$ is the nerve of $\G$, a simplicial set whose $n$-simplices are $(n+1)$-tuples of vertices 
$(v_0,v_1,\ldots, v_n)$,
where there is a path from $v_{k-1}$ to $v_k$ for all~$1\leq k\leq n$. 
For $0\leq i\leq n$, the face operator \[\partial_i\colon 
\Nrv_n(\G)\to \Nrv_{n-1}(\G)\]
deletes the $i$-th vertex, and the degeneracy operator \[s_i\colon 
\Nrv_n(\G)\to \Nrv_{n+1}(\G)\]
duplicates the $i$-th vertex; see~\cite[Definition~3.2]{HR2025}. 
The set of marked edges is~$E$, consisting of all non-degenerate edges $u\to v$ and the degenerate edges \((v,v)\).

\begin{bigremark}
By definition, allowed simplices of \(\mkd{\mathcal{G}}\) correspond to paths in~\(\mathcal{G}\), and hence the path homology of \(\mkd{\mathcal{G}}\) agrees with the GLMY path homology of \(\mathcal{G}\).
The same holds for the corresponding cohomologies.
\end{bigremark}

\subsection{Borel construction}
A discrete group $\Gamma$ is said to act from the left on a directed graph $\mathcal{G}=(V,E)$ if $\Gamma$ acts from the left on the vertex set~$V$ in such a way that directed edges are preserved. That is, for any $(u,v)\in E$ and $\gamma\in \Gamma$, we have 
 $\gamma(u,v)=(\gamma u,\gamma v)\in E.$
Such a graph is called a directed $\Gamma$-graph.

\begin{example}
A directed $\Gamma$-graph can be obtained from a simplicial $\Gamma$-complex as follows.
As described in~\cite{GMY}, one can assign a directed graph $\mathcal{G}_K$ to a simplicial complex $K$, by letting
the vertices of $\mathcal{G}_K$ be the simplices of $K$, and there is a directed edge from $\sigma$ to~$\tau$ when $\tau$ is a codimension-one face of $\sigma$.
If $K$ is a simplicial $\Gamma$-complex (in other words, $\Gamma$ acts on the vertex set of $K$ in a way that it preserves its simplices), then the induced action on simplices turns $\mathcal{G}_K$ into a directed $\Gamma$-graph.
\end{example}

Let $\mathcal{G}=(V,E)$ be a directed $\Gamma$-graph.
Regard $\Gamma$ as a constant simplicial group~$\underline{\Gamma}$, where for any $n\geq 0$, we have $(\underline{\Gamma})_n=\Gamma$, and all face and degeneracy maps are the identity maps. 
Then the marked simplicial set $\mkd{\mathcal{G}}=(\Nrv(\mathcal{G}),E)$ inherits a natural $\underline{\Gamma}$-action. Explicitly, for any simplex $(v_0,\ldots, v_n)\in \Nrv_n(\mathcal{G})$ and any $\gamma\in (\underline{\Gamma})_n=\Gamma$, the action of $\underline{\Gamma}$ on $\Nrv(\mathcal{G})$ is given by
\[\gamma(v_0,\ldots, v_n)=(\gamma v_0,\ldots, \gamma v_n).\]
Moreover, if $(v_0,v_1) \in E$, namely, either $v_0=v_1$ or there is a directed edge from~$v_0$ to $v_1$, then for any $\gamma\in (\underline{\Gamma})_1=\Gamma$, we have
\[
\gamma (v_0,v_1)=(\gamma v_0, \gamma v_1)\in E,
\]
since either $\gamma v_0=\gamma v_1$  or there is a directed edge from $\gamma v_0$ to $\gamma v_1$.

Let $E\underline{\Gamma}\to B\underline{\Gamma}$ be the universal $\underline{\Gamma}$-bundle. 
For simplicity, we shall write $E\Gamma\to B\Gamma$ for this simplicial fibre bundle in what follows.
We will check that the box product of marked simplicial sets $(E{\Gamma})^{\musSharp{}}$ and $(\Nrv(\mathcal{G}), E)$ satisfies the condition of Lemma~\ref{lemma condition for marked borel construction}.

\begin{lemma}
If $v\in V$ and $\gamma\in (\underline{\Gamma})_1=\Gamma$, then $\gamma s_0(v)$ is degenerate. Thus, the pair \[\left(E{\Gamma}\times_{{\Gamma}} \Nrv(\mathcal{G}), \left((E{\Gamma})_1\Box E\right)/\Gamma\right)\] is a marked simplicial set.
\end{lemma}
\begin{proof}
It is evident that  $\gamma s_0(v)=\gamma (v,v)= (\partial_0\gamma v,\partial_0\gamma v)=s_0(\partial_0\gamma v)$. Therefore, Lemma~\ref{lemma condition for marked borel construction} implies 
$\left(E{\Gamma}\times_{{\Gamma}} \Nrv(\mathcal{G}), \left((E{\Gamma})_1\Box E\right)/\Gamma\right)$ is a marked simplicial set.
\end{proof}

\begin{definition}\label{def borel construction for digraphs}
Let $\mathcal{G}$ be a directed $\Gamma$-graph. The \emph{Borel construction} of $\mathcal{G}$ is defined as the marked simplicial set 
\[\mathcal{G}_{\Gamma}\colonequals\left(E{\Gamma}\times_{\Gamma} \Nrv(\mathcal{G}), \left((E{\Gamma})_1\Box E\right)/\Gamma\right).\]  
Define the \emph{(Borel-type) equivariant path
homology and cohomology} of $\mathcal{G}$ to be \[\text{$\PH^\Gamma_*(\mathcal{G})\colonequals \PH_*(\G_\Gamma)$~ and~ $\PH^*_\Gamma(\mathcal{G})\colonequals \PH^*(\G_\Gamma)$.}
\]
\end{definition}

\smallskip

Corollary~\ref{cor model for equivariant homology} provides a natural isomorphism of chain complexes
\[\N(B\Gamma)\otimes_{t_{\sz}}\Omega_\bullet(\mathcal{G})\to \Omega_\bullet(\mathcal{G}_\Gamma).\]
Before we describe the differential of $\N(B\Gamma)\otimes_{t_{\sz}}\Omega_\bullet(\mathcal{G})$, we fix the notation for the induced $\Gamma$-action on path chains. 

The $\Gamma$-action on $\mathcal{G}$ induces a left action on the path chain complex $\Omega_\bullet(\mathcal{G})$
by
\begin{equation}\label{eq gamma action on path chains}
\gamma\cdot \left(\sum_{\p} \alpha_{\p} e_{\p}\right)=
\sum_{\p} \alpha_{\p} e_{\gamma\cdot \p}
\end{equation}
where $\gamma\in \Gamma$ and $\sum_{\p} \alpha_{\p} e_{\p}\in \Omega_\bullet(\G)$ with $\alpha_{\p}\in \kk$ and $\p$ ranging over paths of $\mathcal{G}$.

\begin{proposition}\label{prop model for equivariant homology digraph}
There is an isomorphism of chain complexes
\[
(\N(B\Gamma)\otimes_{t_{\sz}}\Omega_\bullet(\mathcal{G}),\partial_{t_{\sz}})\to(\Omega_\bullet(\mathcal{G}_\Gamma),\partial),
\]
which is natural with respect to any $\Gamma$-equivariant inclusion $\mathcal{G}'\subset \mathcal{G}.$
Moreover, the differential~$\partial_{t_{\sz}}$ has the form
\begin{equation}\label{eq twisted tensor product differential Borel}
\begin{split}
\partial_{t_{\sz}}(x\otimes w)
=
\begin{cases}
\begin{aligned}
[~]\otimes (\partial w) &\qquad \text{ if } x=[~]\in (B\Gamma)_0;
\end{aligned}\\[0.5em]  
\begin{aligned}
(\partial x)&\otimes w+(-1)^{n} x\otimes (\partial w)\\[0.25em]
+&(-1)^n (x_1,\ldots,x_{n-1})\otimes (x_n^{-1}\cdot w-w)\quad \text{if } n\geq 1,
\end{aligned}
\end{cases}
\end{split}
\end{equation}
where $x=(x_1,\ldots,x_n)\in (B\Gamma)_n$, $w\in \Omega_\bullet(\G)$ and $x_n^{-1}\cdot w$ is defined in~\eqref{eq gamma action on path chains}.
\end{proposition}
\begin{proof}
Since $\underline{\Gamma}$ is the constant simplicial group, the operators
$\sz_{\bi}$ defined in~\eqref{eq Szczarba operators} reduce to
\begin{equation}\label{eq Szczarba operators for constant simplicial group}
\sz_{\bi}(x)
=
\sigma(x)\,
\sigma(\partial_0 x)\cdots
\sigma((\partial_0)^{n-1}x),
\quad\text{for } x\in (B\Gamma)_n,\; n\ge 1.
\end{equation}
Consequently, Szczarba’s twisting cochain $t_{\sz}\colon \N_n(B\Gamma)\to \N_{n-1}(\Gamma) $ has the form
\begin{equation}\label{eq Szczarba’s twisting cochain for constant simplicial group}
t_{\sz}(x)=\begin{cases}
    0 & \text{if } n=0;\\
    x_1^{-1}-1 & \text{if } n=1;\\
    \sum_{\bi\in S_{n-1}}(-1)^{|\bi|}\sz_{\bi}(x)=0 & \text{if } n\geq 2.
\end{cases}
\end{equation}
Here for $n=1$, the simplex $x=(x_1)\in (B\Gamma)_1$ is non-degenerate (i.e., $x_1\neq 1$, the identity of $\Gamma$).
For $n\geq 2$, the vanishing of $t_{\sz}(x)$ follows from~\eqref{eq Szczarba operators for constant simplicial group} and the equality $\sum_{\bi\in S_{n-1}}(-1)^{|\bi|}=0$ for $n\ge 2$.

Applying~\eqref{eq twisted tensor product differential} to the twisting cochain~\eqref{eq Szczarba’s twisting cochain for constant simplicial group}
yields the stated formula for the differential $\partial_{t_{\sz}}$.
\end{proof}

\begin{example}\label{ex:borel-flipped-edge}
Assume $\kk=\Z$. Let $\mathcal{G}$ be the directed graph with vertex set $\{0,1\}$ and directed edges
$0\to 1$, $1\to 0.$
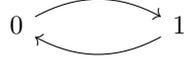
\[
\begin{tikzcd}
0\ar[rr,bend left = 10mm] && 1 \ar[ll,bend left = 10mm]
\end{tikzcd}
\]
Suppose $\Gamma=\mathbb{Z}/2=\langle \alpha\rangle$ acts on $\mathcal{G}$ by 
$\alpha\cdot 0=1 $ and $\alpha\cdot 1=0.$
This action preserves the directed edges and hence defines an action on $\mathcal{G}$.

For each $n\geq 0$, there is a unique non-degenerate $n$-simplex of $B\Z_2$, namely
\begin{equation}\label{eq nondegenerate simplices of BZ2}
x_0=[\,\,],\qquad
x_n=(\alpha,\ldots,\alpha)\in (B\Z_2)_n \text{ with } n\ge 1.
\end{equation}
For $k\geq 0$, let $\p_k$ and $\q_k$ denote the paths of length $k$ starting at the vertex $0$ and the vertex $1$, respectively.
In particular, $\p_0=(0)$ and $\q_0=(1)$.
The $\mathbb{Z}/2$–action exchanges $\p_k$ and $\q_k$.

The path differential~\eqref{eq boundary operator} satisfies
\[
\begin{aligned}
\partial e_{\p_0}&=0, & \partial e_{\q_0}&=0,\\
\partial e_{\p_k}&=e_{\q_{k-1}}+(-1)^k e_{\p_{k-1}}, ~&
\partial e_{\q_k}&=e_{\p_{k-1}}+(-1)^k e_{\q_{k-1}}.
\end{aligned}
\]
Consequently, for each $k\ge 0$, we have
$\Omega_k(\mathcal G)=\Z\cdot\{e_{\p_k},e_{\q_k}\}.$
The equivariant path homology of $\mathcal{G}$ is computed by the chain complex
\[
\N_n(B\Gamma)\otimes_{t_{\sz}}\Omega_k(\mathcal{G}),
\]
with differential given in Proposition~\ref{prop model for equivariant homology digraph}.

If $n=0$ and $k\geq 1$, then
\[\begin{split}
\partial_{t_{\sz}}(x_0\otimes e_{\p_{0}})=0,
\quad
\partial_{t_{\sz}}(x_0\otimes e_{\q_{0}})=0,\\
\partial_{t_{\sz}}(x_0\otimes e_{\p_{k}})=x_0\otimes (e_{\q_{k-1}}+(-1)^k e_{\p_{k-1}}),\\
\partial_{t_{\sz}}(x_0\otimes e_{\q_{k}})=x_0\otimes (e_{\p_{k-1}}+(-1)^ke_{\q_{k-1}}).
\end{split}
\]
If $n,k\geq 1$, then
\[\begin{split}
\partial_{t_{\sz}}(x_n\otimes e_{\p_{0}})&=x_{n-1}\otimes (e_{\p_{0}}+(-1)^ne_{\q_{0}}),\\
\partial_{t_{\sz}}(x_n\otimes e_{\q_{0}})&=x_{n-1}\otimes ((-1)^ne_{\p_{0}}+e_{\q_{0}}),\\
\partial_{t_{\sz}}(x_n\otimes e_{\p_{k}})&=
(-1)^n x_n\otimes ((-1)^{k}e_{\p_{k-1}}+e_{\q_{k-1}})+x_{n-1}\otimes (e_{\p_{k}}+(-1)^ne_{\q_{k}}),
\\
\partial_{t_{\sz}}(x_n\otimes e_{\q_{k}})&=
(-1)^nx_n\otimes 
(e_{\p_{k-1}}+(-1)^{k}e_{\q_{k-1}})
+x_{n-1}\otimes ((-1)^ne_{\p_{k}}+e_{\q_{k}}).
\end{split}
\]


Fix $i\geq 1$. Order the basis of the $i$-th component $(\N(B\Gamma)\otimes \Omega(\mathcal{G}))_i$ as follows:
\[
\{x_0\otimes e_{\p_i}, x_0\otimes e_{\q_i},
x_1\otimes e_{\p_{i-1}}, x_1\otimes e_{\q_{i-1}},\ldots,
x_i\otimes e_{\p_0}, x_i\otimes e_{\q_0}
\}.
\]
With this ordering, the differentials \[\partial_{t_{\sz}}\colon (\N(B\Gamma)\otimes \Omega(\mathcal{G}))_{i}\to (\N(B\Gamma)\otimes \Omega(\mathcal{G}))_{i-1}\] are represented by the matrix~\eqref{matrix odd} if $i$ is odd, and by \eqref{matrix even} if $i$ is even.
Here the two types of matrices are
\begin{equation}\label{matrix odd}
\begin{pmatrix}
    A_1 & -A_1 & 0& 0& 0&\cdots & 0\\
    0& B_2&-B_2 & 0&0&\cdots &0\\
    0 & 0&A_3& -A_3&0&\cdots &0\\
    0&0&0&B_4&-B_4&\cdots &0\\
    \vdots &\vdots &\vdots & \vdots &\ddots &\ddots&\vdots\\
    0 & 0&0& 0&\cdots & A_i &-A_i
\end{pmatrix} 
\end{equation}
and
\begin{equation}\label{matrix even}
\begin{pmatrix}
    A_1 & B_1 & 0& 0& 0&\cdots & 0& 0\\
    0& B_2&A_2 & 0&0&\cdots &0& 0\\
    0 & 0&A_3& B_3&0&\cdots &0& 0\\
    0&0&0&B_4&A_4&\cdots &0& 0\\
    \vdots &\vdots &\vdots & \vdots &\ddots &\ddots&\vdots &\vdots\\
    0 & 0&0& 0&\cdots & A_{i-1} &B_{i-1}& 0\\
    0 & 0&0& 0&\cdots &0 & B_{i} &A_{i}
\end{pmatrix} 
\end{equation}
where  
\[A_\ell=\begin{pmatrix}
    (-1)^i & 1\\
    1&(-1)^i
\end{pmatrix}\quad \text{and} \quad B_\ell=\begin{pmatrix}
    (-1)^i & -1\\
    -1&(-1)^i
\end{pmatrix},\, \text{ $1\leq \ell\leq i$}.\]

A direct computation shows that~\eqref{matrix odd} and \eqref{matrix even}
have Smith norm forms
\[
\begin{pmatrix}
    I_i & 0\\
    0&0
\end{pmatrix}\quad\text{and}\quad
\begin{pmatrix}
    I_i& 0 &0\\
   0 & 2&0\\
    0&0 &0
\end{pmatrix}
\]
where $I_i$ is the identity matrix of rank $i.$
Consequently, we conclude
\[
\PH^\Gamma_i(\mathcal{G};\Z)=\begin{cases}
    \Z &\text{ if }i=0,\\
    \Z/2\Z  &\text{ if $i>0$ is odd},\\
    0 & \text{ otherwise}.
\end{cases}
\]
\end{example}

\begin{example}
We consider the directed graph $\mathcal{G}=(V,E)$ depicted below.
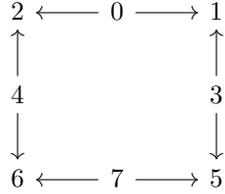
\[
\begin{tikzcd}
    2 &\ar[l] 0\ar[r] &1\\
    4\ar[u]\ar[d] &&3\ar[d]\ar[u] \\
    6&7\ar[l]\ar[r]&5
\end{tikzcd}
\]

Let $\Gamma=\Z_2=\langle\alpha\rangle$.
The action of $\alpha$ on $V$ is swapping the vertices
\[
0\leftrightarrow 7,\qquad
1\leftrightarrow 6,\qquad
2\leftrightarrow 5,\qquad
3\leftrightarrow 4.
\]
This action preserves directed edges, and hence $\mathcal{G}$ is a directed $\Z_2$-graph.

Since $\mathcal{G}$ admits no directed paths of length $\ge 2$, we have
\[
\Omega_i(\mathcal{G})=0 \quad \text{for } i\ge 2.
\]
Hence for $n\geq 0$, we have 
\begin{equation}\label{example 2 basis}
\begin{split}
\N_n(B\Z_2)\otimes \Omega_0(\mathcal{G})&=\kk\cdot \{\,x_n\otimes e_v \mid v\in V\,\},
\\
\N_n(B\Z_2)\otimes \Omega_1(\mathcal{G})&=\kk\cdot \{\,x_n\otimes e_{vw} \mid (v,w)\in E,\; v\ne w\,\}.
\end{split}\end{equation}
We shall apply~\eqref{eq twisted tensor product differential Borel} to the above basis elements.

If $n=0$, the differentials are given by
\[
\partial_{t_{\sz}}(x_0\otimes e_v)=0,
\qquad
\partial_{t_{\sz}}(x_0\otimes e_{vw})=x_0\otimes (e_w-e_v).
\]
Therefore $\PH^\Gamma_0(\mathcal{G};\kk)=\kk.$

If $n\geq 1$, then for $v\in V$, we have
\begin{equation}\label{eq differential vertex}
\partial_{t_{\sz}}(x_n\otimes e_v)=
x_{n-1}\otimes (e_v+(-1)^ne_{\alpha\cdot v})
\end{equation}
and for $(v,w)\in E$ with $v\neq w$, we have
\[
\partial_{t_{\sz}}(x_n\otimes e_{vw})=
(-1)^n x_n\otimes  (e_w-e_v)+
x_{n-1}\otimes (e_{vw}+(-1)^ne_{\alpha\cdot v,\alpha\cdot w}).
\]

The degree~1 cycles are spanned by the following elements
\[
\begin{aligned}
& e_{02}-e_{42}+e_{46}-e_{76}+e_{75}-e_{35}+e_{31}-e_{01},\\
& x_1\otimes (e_0+e_7),\;
x_1\otimes (e_1+e_6),\;
x_1\otimes (e_2+e_5),\;
x_1\otimes (e_3+e_4).
\end{aligned}
\]
By~\eqref{eq differential vertex}, the latter four cycles $x_1\otimes (x_v+x_{\alpha\cdot v})$ are boundaries, so the first cycle generates the homology.
Hence
$\PH_1^\Gamma(\mathcal{G};\kk)\cong \kk$.

\smallskip

When $n\geq 2$, the differential
\[\partial_{t_{\sz}}\colon
\N_{n-i}(B\Z_2)\otimes \Omega_i(\mathcal{G})
\xrightarrow{}
\N_{n-i-1}(B\Z_2)\otimes \Omega_i(\mathcal{G})
\oplus
\N_{n-i}(B\Z_2)\otimes \Omega_{i-1}(\mathcal{G})
\]
is represented by a
$16\times 16$ integer matrix with respect to the bases~\eqref{example 2 basis}. Its
Smith normal form is
\[
\begin{pmatrix}
I_8 & 0\\
0 & 0
\end{pmatrix}.
\]
Consequently, we conclude
\[
\PH_n^\Gamma(\mathcal{G};\kk)=
\begin{cases}
\kk & \text{ if  } n=0,1;\\
 0 & \text{ otherwise}.   
\end{cases}
\]
\end{example}

\subsection{The dual case}
In this subsection, we work over a field $\FF$ and assume $\mathcal{G}$ is a finite directed $\Gamma$-graph, where $\Gamma$ is finite.
Let $\Omega^\bullet(\mathcal{G})$ and~$\Omega^\bullet(\mathcal{G}_\Gamma)$ be the path cochain algebra over $\FF$. Then Proposition~\ref{prop dual} implies
\begin{equation}\label{eq dual}
\Omega^\bullet(\mathcal{G})\cong \Omega_\bullet(\mathcal{G})^\vee\quad\text{and}\quad
\Omega^\bullet(\mathcal{G}_\Gamma)\cong \Omega_\bullet(\mathcal{G}_\Gamma)^\vee
\end{equation}
where 
$\Omega_\bullet(\mathcal{G})^\vee=\hom_\FF(\Omega_\bullet(\mathcal{G}),\FF)$ and $\Omega_\bullet(\mathcal{G}_\Gamma)^\vee=\hom_\FF(\Omega_\bullet(\mathcal{G}_\Gamma),\FF)$.

Since each $\N_k(B\Gamma)$ and $\Omega_\ell(\mathcal{G})$ are finitely generated (due to the assumption $\Gamma$ and $\mathcal{G}$ are finite), the dual of $\N(B\Gamma)\otimes_{t_{\sz}}\Omega_\bullet(\mathcal{G})$
identifies with the tensor product of the
duals, i.e.,
\[
\bigl(\N(B\Gamma)\otimes\Omega_\bullet(\mathcal{G})\bigr)^\vee
\cong
\N^\bullet(B\Gamma)\otimes\Omega_\bullet(\mathcal{G})^\vee .
\]
Together with the identifications~\eqref{eq dual}, we obtain an isomorphism
\begin{equation}\label{eq identification}
\bigl(\N(B\Gamma)\otimes\Omega_\bullet(\mathcal{G})\bigr)^\vee\cong \N^\bullet(B\Gamma)\otimes\Omega^\bullet(\mathcal{G}).
\end{equation}

To describe the dual differential on $\N^\bullet(B\Gamma)\otimes\Omega^\bullet(\mathcal{G})$, we first define the dual of the~$\Gamma$-action on $\Omega_\bullet(\mathcal{G})$ given in~\eqref{eq gamma action on path chains}. 
Let $\theta\colon \Omega^\bullet(\mathcal{G})\to \Omega_\bullet(\mathcal{G})^\vee$ denote the first isomorphism in~\eqref{eq dual}.
For $\gamma\in \Gamma$ and $u\in \Omega^\bullet(\mathcal{G})$, 
define $\gamma\wedge u\in \Omega^\bullet(\mathcal{G})$ by the formula
\[
\theta(\gamma\wedge u) (w)=\theta(u)(\gamma\cdot w),\quad w\in \Omega_\bullet(\mathcal{G}).
\]

For a non-degenerate $k$-simplex $x=(x_1,\ldots, x_k)$, let $x^*\in \N^{k}(B\Gamma)$ be the dual basis element corresponding to $x$.
The dual of~\eqref{eq twisted tensor product differential Borel} has the form
\begin{multline}
\label{eq the dual differential}
\delta_{t_{\sz}}(x^*\otimes u)=\delta(x^*)\otimes u+(-1)^kx\otimes \partial(u)\\
+\sum_{h\in \Gamma\setminus\{1\}}(-1)^k(x_1,\ldots,x_k,h)^*\otimes (h^{-1}\wedge u-u).
\end{multline}

We denote by $\N^\bullet(B\Gamma)\otimes_{t_{\sz}}\Omega^\bullet(\mathcal{G})$ for the cochain complex $\N^\bullet(B\Gamma)\otimes \Omega^\bullet(\mathcal{G})$ equipped with the differential $\delta_{t_{\sz}}$.


\begin{proposition}
\label{prop Cohomological model}
Suppose that $\mathcal{G}$ is a finite directed $\Gamma$-graph, where $\Gamma$ is finite. Over a field $\FF$, there is an isomorphism of cochain complexes
\[
\Omega^\bullet(\mathcal{G}_\Gamma)
\xrightarrow{}
\N^\bullet(B\Gamma)\otimes_{t_{\sz}}\Omega^\bullet(\mathcal{G}),
\]
which is natural with respect to  $\Gamma$-equivariant inclusions of directed graphs.
\end{proposition}
\begin{proof}
By Proposition~\ref{prop model for equivariant homology digraph}, there is an
isomorphism of chain complexes
\begin{equation}\label{eq chain isomorphism equivariant homology digraph}
\N(B\Gamma)\otimes_{t_{\sz}}\Omega_\bullet(\mathcal{G})
\xrightarrow{\;\cong\;}
\Omega_\bullet(\mathcal{G}_\Gamma).
\end{equation}
Since $\FF$ is a field, the dual of a linear isomorphism is again a linear
isomorphism, with inverse given by dualizing the inverse map. Hence applying
the contravariant functor $\hom_{\FF}(-,\FF)$ yields an isomorphism of cochain
complexes
\[
\Omega_\bullet(\mathcal{G}_\Gamma)^\vee
\;\xrightarrow{\;\cong\;}
\bigl(\N(B\Gamma)\otimes_{t_{\sz}}\Omega_\bullet(\mathcal{G})\bigr)^\vee .
\]

The right-hand side is identified with $\N^\bullet(B\Gamma)\otimes_{t_{\sz}}\Omega^\bullet(\mathcal{G})$ due to the isomorphism of~\eqref{eq identification}, under which the
dual differential $\delta_{t_{\sz}}$ is described in
\eqref{eq the dual differential}.
The left-hand side is isomorphic to~$\Omega^\bullet(\mathcal{G}_\Gamma)$, which completes the proof.

Naturality with respect to $\Gamma$-equivariant inclusions of directed graphs
follows from the naturality of~\eqref{eq chain isomorphism equivariant homology digraph} and of the dualisation .
\end{proof}

\subsection{The relative case}
Let $\mathcal{G}'\subset \mathcal{G}$ be a $\Gamma$-invariant sub-directed graph. The relative Borel-type equivariant path homology of the pair $(\mathcal{G},\mathcal{G}')$ is defined as the homology of the quotient complex
\begin{equation}\label{eq relative chain complex}
\Omega_\bullet(\mathcal{G}_\Gamma,\mathcal{G}'_\Gamma)\colonequals \Omega_\bullet(\mathcal{G}_\Gamma)/\Omega_\bullet(\mathcal{G}'_\Gamma),
\end{equation}
and is denoted by $\PH^\Gamma_*(\mathcal{G},\mathcal{G}')$.

Proposition~\ref{prop model for equivariant homology digraph}
has the following relative version.

\begin{proposition}\label{prop relative version} Let $\kk$ be a principal ideal domain.
There is an isomorphism of chain complexes
\[
\N(B\Gamma)\otimes_{t_{\sz}}\Omega_\bullet(\mathcal{G},\mathcal{G}')\longrightarrow
\Omega_\bullet(\mathcal{G}_\Gamma,\mathcal{G}'_\Gamma).
\]
\end{proposition}

\begin{proof}
The inclusion $\mathcal{G}'\subset \mathcal{G}$ induces a short exact sequence of chain complexes
\[
0 \longrightarrow \Omega_\bullet(\mathcal{G}') \longrightarrow \Omega_\bullet(\mathcal{G})
\longrightarrow \Omega_\bullet(\mathcal{G},\mathcal{G}') \longrightarrow 0.
\]

Each $\N_n(B\Gamma)$ is a free $\kk$-module and $\N(B\Gamma)$ is a direct summand of $\N_n(B\Gamma)$'s.
Therefore the ordinary tensor product functor $\N(B\Gamma)\otimes(-)$ is exact. 
Applying this to the above sequence, we obtain a short exact
sequence
\begin{equation}\label{eq short exact sequence 1}
0 \longrightarrow
\N(B\Gamma)\otimes\Omega_\bullet(\mathcal{G}') 
\longrightarrow
\N(B\Gamma)\otimes\Omega_\bullet(\mathcal{G})
\longrightarrow
\N(B\Gamma)\otimes\Omega_\bullet(\mathcal{G},\mathcal{G}') 
\longrightarrow 0.
\end{equation}

On the other hand, the Borel construction gives a short exact sequence
\begin{equation}\label{eq short exact sequence 2}
0 \longrightarrow
\Omega_\bullet(\mathcal{G}'_\Gamma) \longrightarrow
\Omega_\bullet(\mathcal{G}_\Gamma) \longrightarrow
\Omega_\bullet(\mathcal{G}_\Gamma,\mathcal{G}'_\Gamma) \longrightarrow 0.
\end{equation}
By Proposition~\ref{prop model for equivariant homology digraph}, the isomorphism
\[
\N(B\Gamma)\otimes_{t_{\sz}}\Omega_\bullet(\mathcal{G}) \to \Omega_\bullet(\mathcal{G}_\Gamma)
\]
is natural with respect to $\Gamma$-equivariant inclusions of directed graphs.
Applying this to the inclusion $G'\hookrightarrow G$, we obtain a commutative diagram of chain complexes
\[\begin{tikzcd} \N(B\Gamma)\otimes_{t_{\sz}}\Omega_\bullet(\mathcal{G}') \arrow[r,""] \arrow[d] & \Omega_\bullet(\mathcal{G}'_\Gamma) \arrow[d] \\ \N(B\Gamma)\otimes_{t_{\sz}}\Omega_\bullet(\mathcal{G}) \arrow[r,""] & \Omega_\bullet(\mathcal{G}_\Gamma), \end{tikzcd}\]
where the horizontal maps are quasi-isomorphisms and the vertical maps are
induced by inclusion. Passing to quotients, we obtain an induced morphism
\[
\N(B\Gamma)\otimes_{t_{\sz}}\big(\Omega_\bullet(\mathcal{G})/\Omega_\bullet(\mathcal{G}')
\big)
\to \Omega_\bullet(\mathcal{G}_\Gamma)/\Omega_\bullet(\mathcal{G}'_\Gamma),
\]
which is the asserted map
\begin{equation}\label{eq induced map}
\N(B\Gamma)\otimes_{t_{\sz}}\Omega_\bullet(\mathcal{G},\mathcal{G}')\to\Omega_\bullet(\mathcal{G}_\Gamma,\mathcal{G}'_\Gamma).
\end{equation}

The long exact sequence in homology associated
to~\eqref{eq short exact sequence 1} and~\eqref{eq short exact sequence 2} imply that the induced map~\eqref{eq induced map}
is also a quasi-isomorphism.
\end{proof}

Let $\mathcal{G}'\subset \mathcal{G}$ be a $\Gamma$-invariant directed subgraph and let $\kk=\FF$ be a field.
Define the relative cochain complexes by
\[\begin{split}
\Omega^\bullet(\mathcal{G},\mathcal{G}')&=\ker(\Omega_\bullet(\mathcal{G})^\vee\to \Omega_\bullet(\mathcal{G}')^\vee),\\
\Omega^\bullet(\mathcal{G}_\Gamma,\mathcal{G}'_\Gamma)&=\ker(\Omega_\bullet(\mathcal{G}_\Gamma)^\vee\to \Omega_\bullet(\mathcal{G}'_\Gamma)^\vee),
\end{split}\]
where $(-)^\vee$ is the Hom-dual over $\FF$.
Since $\FF$ is a field, these complexes are canonically isomorphic to the dual complexes
\[
\text{$\Omega^\bullet(\mathcal{G},\mathcal{G}')\cong \Omega_\bullet(\mathcal{G},\mathcal{G}')^\vee$ and 
$\Omega^\bullet(\mathcal{G}_\Gamma,\mathcal{G}'_\Gamma)\cong \Omega_\bullet(\mathcal{G}_\Gamma,\mathcal{G}'_\Gamma)^\vee$.}\]

\begin{proposition}
Under the same assumptions of Proposition~\ref{prop Cohomological model}, there is an isomorphism of cochain complexes
\[
\Omega^\bullet(\mathcal{G}_\Gamma,\mathcal{G}'_\Gamma)\to
\N^\bullet(B\Gamma)\otimes_{t_{\sz}}\Omega^\bullet(\mathcal{G},\mathcal{G}'),
\]
where the right-hand side denotes the tensor product $\N^\bullet(B\Gamma)\otimes \Omega^\bullet(\mathcal{G},\mathcal{G}')$ equipped the differential induced from~\eqref{eq the dual differential}.
\end{proposition}
\begin{proof}
By Proposition~\ref{prop relative version},
there is an isomorphism of chain complexes
\[
\N(B\Gamma)\otimes_{t_{\sz}}\Omega_\bullet(\mathcal{G},\mathcal{G}')
\xrightarrow{}
\Omega_\bullet(\mathcal{G}_\Gamma,\mathcal{G}'_\Gamma).
\]
Applying $\hom_\FF(-,\FF)$ and using exactness over a filed yields an isomorphism
\[
\Omega_\bullet(\mathcal{G}_\Gamma,\mathcal{G}'_\Gamma)^\vee
\xrightarrow{}
\bigl(\N(B\Gamma)\otimes_{t_{\sz}}\Omega_\bullet(\mathcal{G},\mathcal{G}')\bigr)^\vee.
\]
Since $\Gamma$ and $\G$ are finite by assumption, each $\N_k(B\Gamma)$ and $\Omega_\ell(\mathcal{G},\mathcal{G}')$ are finitely generated. 
Therefore, the dual of tensor products is identified with the tensor
product of duals. Together with the identification $\Omega^\bullet(\mathcal{G}_\Gamma,\mathcal{G}'_\Gamma)\cong \Omega_\bullet(\mathcal{G}_\Gamma,\mathcal{G}'_\Gamma)^\vee$, this gives the result.
\end{proof}

Suppose that there are no directed edges from any vertex in $U=V_{\mathcal{G}}\setminus V_{\mathcal{G}'}$ to any vertices of~$\mathcal{G}'$. Then \cite[Lemma 3.10]{GJMY} 
shows that there is an isomorphism 
\begin{equation}\label{eq quotient chain complexes}
\Omega_\bullet(\mathcal{G},\mathcal{G}')\cong \Omega_\bullet^{U}(\mathcal{G})
\end{equation}
where 
$\Omega_\bullet^{U}(\mathcal{G})=\Omega_\bullet(\mathcal{G})\cap \A^U(\mathcal{G})$, and
 $\A^U(\mathcal{G})$ denotes the $\kk$-module generated by all paths $(v_0,\ldots, v_n)$ such that at least one of the vertices $v_i$ lies in $U$.

\smallskip

Substituting~\eqref{eq quotient chain complexes} into Proposition~\ref{prop relative version} gives the following result.

\begin{corollary}\label{cor:relative-supported-homology}
Let $\mathcal{G}$, $\mathcal{G}'$, and $U$ be as above.
\begin{enumerate}
\item If $\kk$ is a principal ideal domain, then
there is an isomorphism of chain complexes
\[
\N(B\Gamma)\otimes_{t_{\sz}}\Omega_\bullet^{U}(\mathcal{G})
\xrightarrow{}
\Omega_\bullet(\mathcal{G}_\Gamma,\mathcal{G}'_\Gamma).
\]
\item 
If $\kk=\FF$ is a field, then
there is an isomorphism of cochain complexes
\[
\Omega^\bullet(\mathcal{G}_\Gamma,\mathcal{G}'_\Gamma)
\xrightarrow{}
\N^\bullet(B\Gamma) \otimes_{t_{\sz}} \Omega_\bullet^{U}(\mathcal{G})^\vee.
\]
\end{enumerate}

\end{corollary}
\begin{proof}
This follows immediately from Proposition~\ref{prop relative version},
the identification~\eqref{eq quotient chain complexes} and its dual.
\end{proof}

\bibliographystyle{plain}

\end{document}